\documentclass[11pt,reqno,twoside]{article}
\usepackage{mysty_arxiv}

\usepackage{wrapfig}

\usepackage{fullpage}
\usepackage{patchcmd}
\makeatletter
\patchcommand\@starttoc{\begin{quote}}{\end{quote}}
\makeatother

\newcommand{\cgal}{CGALP~}

\usepackage{enumitem}

\SetLabelAlign{parright}{\parbox[t]{\labelwidth}{\raggedleft{#1}}}
\setlist[description]{style=multiline,topsep=4pt,align=parright}

\makeatletter
\let\reftagform@=\tagform@
\def\tagform@#1{\maketag@@@{(\ignorespaces\textcolor{black}{#1}\unskip\@@italiccorr)}}
\newcommand{\iref}[1]{\textup{\reftagform@{\tcr{\ref{#1}}}}}
\makeatother

\begin{document}
\title{Generalized Conditional Gradient with Augmented Lagrangian for Composite Minimization}
\author{Antonio Silveti-Falls\thanks{Normandie Universit\'e, ENSICAEN, UNICAEN, CNRS, GREYC, France. E-mail: tonys.falls@gmail.com, cecio.molinari@gmail.com, Jalal.Fadili@ensicaen.fr.} \and
Cesare Molinari\samethanks \and 
Jalal Fadili\samethanks
}
\date{}
\maketitle
\begin{flushleft}\end{flushleft}
\begin{abstract}
In this paper we propose a splitting scheme which hybridizes generalized conditional gradient with a proximal step which we call \cgal algorithm, for minimizing the sum of three proper convex and lower-semicontinuous functions in real Hilbert spaces. The minimization is subject to an affine constraint, that allows in particular to deal with composite problems (sum of more than three functions) in a separate way by the usual product space technique. While classical conditional gradient methods require Lipschitz-continuity of the gradient of the differentiable part of the objective, \cgal needs only differentiability (on an appropriate subset), hence circumventing the intricate question of Lipschitz continuity of gradients. For the two remaining functions in the objective, we do not require any additional regularity assumption. The second function, possibly nonsmooth, is assumed simple, i.e., the associated proximal mapping is easily computable. For the third function, again nonsmooth, we just assume that its domain is weakly compact and that a linearly perturbed minimization oracle is accessible. In particular, this last function can be chosen to be the indicator of a nonempty bounded closed convex set, in order to deal with additional constraints. Finally, the affine constraint is addressed by the augmented Lagrangian approach. Our analysis is carried out for a wide choice of algorithm parameters satisfying so called "open loop" rules. As main results, under mild conditions, we show asymptotic feasibility with respect to the affine constraint, boundedness of the dual multipliers, and convergence of the Lagrangian values to the saddle-point optimal value. We also provide (subsequential) rates of convergence for both the feasibility gap and the Lagrangian values.
\end{abstract}

\begin{keywords}
Conditional gradient; Augmented Lagrangian; Composite minimization; Proximal mapping; Moreau envelope.
\end{keywords}

\begin{AMS}
49J52, 65K05, 65K10.
\end{AMS}


\section{Introduction}\label{sec:intro}

\begin{subsection}{Problem Statement}

In this work, we consider the composite optimization problem,
\begin{equation}\label{PProb}\tag{$\mathrsfs{P}$}
\min\limits_{x\in\HH_p} \brac{f(x) + g(Tx) + h(x): \ Ax=b},
\end{equation}
where $\mc{H}_p, \mc{H}_d, \mc{H}_v$ are real Hilbert spaces (the subindices $p, d$ and $v$ denoting the \textquotedblleft primal\textquotedblright, the \textquotedblleft dual\textquotedblright \  and an auxiliary space - respectively), endowed with the associated scalar products and norms (to be understood from the context), $A:\HH_p\to\HH_d$ and $T:\HH_p\to\HH_v$ are bounded linear operators, $b\in\HH_d$ and $f$, $g$, $h$ are proper, convex, and lower semi-continuous functions with $\C \eqdef \dom\para{h}$ being a weakly compact subset of $\HH_p$. We allow for some \emph{asymmetry} in regularity between the functions involved in the objective. While $g$ is assumed to be prox-friendly, for $h$ we assume that it is easy to compute a linearly-perturbed oracle (see \eqref{oracle}). On the other hand, $f$ is assumed to be differentiable and satisfies a condition that generalizes Lipschitz-continuity of the gradient (see Definition \ref{def_smooth}).

Problem \eqref{PProb} can be seen as a generalization of the classical Frank-Wolfe problem in \cite{frankwolfe} of minimizing a Lipschitz-smooth function $f$ on a convex closed bounded subset $\C \subset \HH_p$,
\nnewq{\label{cfwProb}
\min\limits_{x\in \HH_p}\brac{f(x): \ x\in \C}
}
In fact, if $A\equiv 0$, $b\equiv 0$, $g\equiv 0$, and $h \equiv \iota_{\C}$ is the indicator function of $\C$ then we recover exactly \eqref{cfwProb} from \eqref{PProb}.
\end{subsection}

\begin{subsection}{Contribution}
We develop and analyze a novel algorithm to solve \eqref{PProb} which combines penalization for the nonsmooth function $g$ with the augmented Lagrangian method for the affine constraint $Ax=b$. In turn, this achieves full splitting of all the parts in the composite problem~\eqref{PProb} by using the proximal mapping of $g$ (assumed prox-friendly) and a linear oracle for $h$ of the form~\eqref{oracle}. Our analysis shows that the sequence of iterates is asymptotically feasible for the affine constraint, that the sequence of dual variables converges weakly to a solution of the dual problem, that the associated Lagrangian converges to optimality, and establishes convergence rates for a family of sequences of step sizes and sequences of smoothing/penalization parameters which satisfy so-called "open loop" rules in the sense of \cite{polak} and \cite{dunn}. This means that the allowable sequences of parameters do not depend on the iterates, in contrast to a "closed loop" rule, e.g. line search or other adaptive step sizes. Our analysis also shows, in the case where \eqref{PProb} admits a unique minimizer, weak convergence of the whole sequence of primal iterates to the solution. 

The structure of \eqref{PProb} generalizes \eqref{cfwProb} in several ways. First, we allow for a possibly nonsmooth term $g$. Second, we consider $h$ beyond the case of an indicator function where the linear oracle of the form
\begin{equation}\label{oracle}
\min\limits_{s\in\HH}h\para{s}+\ip{x,s}{}
\end{equation}
can be easily solved. Observe that~\eqref{oracle} has a solution over $\dom(h)$ since the latter is weakly compact. This oracle is reminiscent of that in the generalized conditional gradient method~\cite{Bredies08,Bredies2009,Beck15,Bach15}. Third, the regularity assumptions on $f$ are also greatly weakened to go far beyond the standard Lipschitz gradient case. Finally, handling an affine constraint in our problem means that our framework can be applied to the splitting of a wide range of composite optimization problems, through a product space technique, including those involving finitely many functions $h_i$ and $g_i$, and, in particular, intersection of finitely many nonempty bounded closed convex sets; see~\secref{sec:app}.
\end{subsection}

\begin{subsection}{Relation to prior work}

In the 1950's Frank and Wolfe developed the so-called Frank-Wolfe algorithm in \cite{frankwolfe}, also commonly referred to as the conditional gradient algorithm \cite{levitin,Demyanov,dunn}, for solving problems of the form \eqref{cfwProb}. The main idea is to replace the objective function $f$ with a linear model at each iteration and solve the resulting linear optimization problem; the solution to the linear model is used as a step direction and the next iterate is computed as a convex combination of the current iterate and the step direction. We generalize this setting to include composite optimization problems involving both smooth and nonsmooth terms, intersection of multiple constraint sets, and also affine constraints.

Frank-Wolfe algorithms have received a lot of attention in the modern era due to their effectiveness in fields with high-dimensional problems like machine learning and signal processing (without being exhaustive, see, e.g., \cite{jaggi2013,beck2004,lacoste2015,harchaoui,zhang2012,narasimhan,catala}). In the past, composite, constrained problems like \eqref{PProb} have been approached using proximal splitting methods, e.g. generalized forward-backward as developed in \cite{gfb2011} or forward-douglas-rachford \cite{cesarefdr}. Such approaches require one to compute the proximal mapping associated to the function $h$. Alternatively, when the objective function satisfies some regularity conditions and when the constraint set is well behaved, one can forgo computing a proximal mapping, instead computing a linear minimization oracle. The computation of the proximal step can be prohibitively expensive; for example, when $h$ is the indicator function of the nuclear norm ball, computing the proximal operator of $h$ requires a full singular value decomposition while the linear minimization oracle over the nuclear norm ball requires only the leading singular vector to be computed (\cite{jaggi2010simple}, \cite{yurtsever2017}). Unfortunately, the regularity assumptions required by classical Frank-Wolfe style algorithms are too restrictive to apply to general problems like \eqref{PProb}.

While finalizing this work, we became aware of the recent work of \cite{Yurt}, who independently developed a conditional gradient-based framework which allows one to solve composite optimization problems involving a Lipschitz-smooth function $f$ and a nonsmooth function $g$,
\nnewq{
\min\limits_{x\in \C}\brac{f(x) + g\para{Tx}}.
}
The main idea is to replace $g$ with its Moreau envelope of index $\beta_k$ at each iteration $k$, with the index parameter $\beta_k$ going to $0$. This is equivalent to partial minimization with a quadratic penalization term, as in our algorithm. Like our algorithm, that of~\cite{Yurt} is able to handle problems involving both smooth and nonsmooth terms, intersection of multiple constraint sets and affine constraints, however their algorithms employ different methods for these situations. Our algorithm uses an augmented Lagrangian to handle the affine constraint while the conditional gradient framework treats the affine constraint as a nonsmooth term $g$ and uses penalization to smooth the indicator function corresponding to the affine constraint. In particular circumstances, outlined in more detail in \secref{sec:comparison}, our algorithms agree completely.

Another recent and parallel work to ours is that of \cite{Gidel}, where the Frank-Wolfe via Augmented Lagrangian (FW-AL) is developed to approach the problem of minimizing a Lipschitz-smooth function over a convex, compact set with a linear constraint,
\nnewq{
\min\limits_{x \in \C}\brac{f(x) : \ Ax=0}.
}
The main idea of FW-AL is to use the augmented Lagrangian to handle the linear constraint and then apply the classical augmented Lagrangian algorithm, except that the marginal minimization on the primal variable that is usually performed is replaced by an inner loop of Frank-Wolfe. It turns out that the problem they consider is a particular case of \eqref{PProb}, discussed in \secref{sec:comparison}.
\end{subsection}

\begin{subsection}{Organization of the paper}
In \secref{sec:notation} we introduce the notation and review some necessary material from convex and real analysis. In Section \ref{sec:alg} we present the \textbf{C}onditional \textbf{G}radient with \textbf{A}ugmented \textbf{L}agrangian and \textbf{P}roximal-step \textbf{(\cgal)} algorithm and the underlying assumptions. In Section \ref{sec:convan}, we first state our main convergence results and then turn to their proof. The latter is divided in three main parts. First we show the asymptotic feasibility, then the boundedness of the dual multiplier in the augmented Lagrangian and finally the optimality guarantees, i.e. weak convergence of the sequence $\seq{\mu_k}$ to a solution of the dual problem, weak subsequential convergence of the sequence $\seq{x_k}$ to a solution of the primal problem, and convergence of the Lagrangian values, and  with convergence rates. In \secref{sec:app} we describe how our framework can be instantiated to solve a variety of composite optimization problems. In \secref{sec:comparison} we provide a more detailed discussion comparing \cgal to prior work. Some numerical results are reported in Section~\ref{sec:num}. 

For readers who are primarily interested in the practical perspective, we suggest skipping directly to \secref{sec:alg} for the algorithms and its assumptions or  \secref{sec:convan} for the main convergence results.
\end{subsection}

\section{Notation and Preliminaries}\label{sec:notation}


We first recall some important definitions and results from convex analysis. For a more comprehensive coverage we refer the interested reader to \cite{BAUSCHCOMB,PEYPOU} and \cite{rockafellar1997convex} in the finite dimensional case. Throughout, we let $\mc{H}$ denote an arbitrary real Hilbert space and $g$ an arbitrary function from $\mc{H}$ to the real extended line, namely $g:\mc{H}\to\R\cup\brac{\pinfty}$. The function $g$ is said to belong to $\Gamma_0\para{\mc{H}}$ if it is proper, convex, and lower semi-continuous. The \emph{domain} of $g$ is defined to be $\dom\para{g}\eqdef  \brac{x\in\mc{H}: g\para{x}<\pinfty}$. The \emph{Legendre-Fenchel conjugate} of $g$ is the function $g^{*}:\mc{H}\to\R\cup\brac{\pinfty}$ such that, for every $u\in\mc{H}$,
\newq{
	g^{*}\para{u} \eqdef  \sup\limits_{x\in\mc{H}} \brac{\ip{u,x}{}-g\para{x}}.
} 
Notice that
\begin{equation}\label{fact}
		g_1 \leq g_2 \quad \implies \quad g_2^{*} \leq g_1^{*}.
\end{equation}

\paragraph{Moreau proximal mapping and envelope}
The \emph{proximal operator} for the function $g$ is defined to be
\newq{
	\prox_{g}\para{x} \eqdef  \argmin\limits_{y\in\mc{H}}\brac{g(y) + \frac{1}{2}\norm{x-y}{}^2}
}
and its \emph{Moreau envelope} with parameter $\beta$ as
\begin{equation}\label{moreau_env}
	g^{\beta}\para{x} \eqdef  \inf\limits_{y\in\mc{H}}\brac{g(y) + \frac{1}{2\beta}\norm{x-y}{}^2}.
\end{equation}
Denoting $x^+=\prox_{g}\para{x}$, we have the following classical inequality (see, for instance, \cite[Chapter 6.2.1]{PEYPOU}): for every $y\in\HH$,
\begin{equation}\label{non-exp}
2\left[g(x^+)-g(y)\right]+\|x^+-y\|^2-\norm{x^{}-y}{}^2+\|x^+-x\|^2\leq 0.
\end{equation}
We recall that the \emph{subdifferential} of the function $g$ is defined as the set-valued operator $\partial g: \mc{H} \to 2^{\mc{H}}$ such that, for every $x$ in $\mc{H}$,
\begin{equation}\label{subdiff}
\partial g(x)=\enscond{u \in \mc{H}}{g(y)\geq g(x)+\ip{u, y-x}{} \quad \forall y\in\mc{H}}.
\end{equation}
We denote $\dom(\p g) \eqdef \enscond{x \in \HH}{\p g(x) \neq \emptyset}$. When $g$ belongs to $\Gamma_0\para{\mc{H}}$, it is well-known that the subdifferential is a maximal monotone operator.
If, moreover, the function is G\^ateaux differentiable at $x\in\mc{H}$, then $\partial g(x)=\left\{\nabla g(x)\right\}$. For $x \in \dom(\p g)$, the \emph{minimal norm selection} of $\partial g(x)$ is defined to be the unique element $\brac{\sbrac{\partial g\para{x}}^0} \eqdef \Argmin\limits_{y\in \partial g \para{x}}\norm{y}{}$. Then we have the following fundamental result about Moreau envelopes.

\begin{proposition}{}\label{prop:moreau}
Given a function $g\in\Gamma_0\para{\HH}$, we have the following:
\begin{enumerate}[label=(\roman*)]
\item\label{moreauclimconv} The Moreau envelope is convex, real-valued, and continuous.
\item\label{moreauclaim1} Lax-Hopf formula: the Moreau envelope is the viscosity solution to the following Hamilton Jacobi equation: 
\begin{equation}
\label{LH}
\begin{cases}\frac{\p}{\p \beta}g^{\beta}\para{x} = -\frac{1}{2}\norm{\nabla_x g^{\beta}\para{x}}{}^2 \quad & \para{x,\beta}\in \HH\times (0,\pinfty)\\ g^0\para{x} = g\para{x}\quad & x\in\HH. \end{cases}
\end{equation}
\item\label{moreauclaim2} The gradient of the Moreau envelope is $\frac{1}{\beta}$-Lipschitz continuous and is given by the expression
\newq{
	\nabla_x g^{\beta}\para{x} = \frac{x-\prox_{\beta g}\para{x}}{\beta}.
}
\item\label{moreauclaim3} $\forall x \in \dom(\p g)$, $\norm{\nabla g^{\beta}\para{x}}{} \nearrow \norm{\sbrac{\p g\para{x}}^0}{}$ as $\beta\searrow 0$.
\item\label{moreauclaim4} $\forall x \in \HH$, $g^{\beta}(x) \nearrow g(x)$ as $\beta\searrow 0$. In addition, given two positive real numbers $\beta'<\beta$, for all $x\in \HH$ we have
\begin{equation*}
\begin{split}
0 \leq g^{\beta'}\para{x}-g^{\beta}\para{x} & \leq \frac{\beta-\beta'}{2}\norm{\nabla_x g^{\beta'}\para{x}}{}^2;\\
0 \leq g\para{x}-g^{\beta}\para{x} & \leq \frac{\beta}{2}\norm{\sbrac{\p g\para{x}}^0}{}^2.
\end{split}
\end{equation*}
\end{enumerate}
\end{proposition}
\begin{proof}
\ref{moreauclimconv}: see \cite[Proposition~12.15]{BAUSCHCOMB}. The proof for \ref{moreauclaim1} can be found in \cite[Lemma~3.27 and Remark~3.32]{attouch} (see also~\cite{Imbert01} or \cite[Section~3.1]{ambrosiobook}). The proof for claim \ref{moreauclaim2} can be found in \cite[Proposition 12.29]{BAUSCHCOMB} and the proof for claim \ref{moreauclaim3} can be found in \cite[Corollary 23.46]{BAUSCHCOMB}. For the first part in \ref{moreauclaim4}, see~\cite[Proposition~12.32(i)]{BAUSCHCOMB}. To show the first inequality in \ref{moreauclaim4}, combine \ref{moreauclaim1} and convexity of the function $\beta \mapsto g^{\beta}\para{x}$ for every $x \in \HH$. The second inequality follows from the first one and \ref{moreauclaim3}, taking the limit as $\beta'\to 0$.
\end{proof}

\begin{remark}
$~$\\\vspace*{-0.5cm}
\begin{enumerate}[label=(\roman*)]
\item While the regularity claim in \propref{prop:moreau}\ref{moreauclaim2} of the Moreau envelope $g^{\beta}\para{x}$ w.r.t. $x$ is well-known, a less known result is the $C^1$-regularity w.r.t. $\beta$ for any $x \in \HH$ (\propref{prop:moreau}\ref{moreauclaim1}). To our knowledge, the proof goes back, at least, to the book of \cite{attouch}. Though it has been rediscovered in the recent literature in less general settings.
	
\item For given functions $H: \HH \to \R$ and $g_0: \HH \to \R$, a natural generalization of the Hamilton-Jacobi equation in \eqref{LH} is
	\begin{equation*}
	\begin{cases}\frac{\p}{\p \beta}g\para{x,\beta} +H\left(\nabla_x g\para{x,\beta}\right)=0 & \para{x,\beta}\in \HH\times (0,\pinfty)\\ 
	g\para{x,0} = g\para{x}\quad & x\in\HH.
	\end{cases}
	\end{equation*}
	Supposing that $H$ is convex and that $\lim\limits_{\norm{p}{}\to \pinfty} H(p)/\norm{p}{} = \pinfty$, the solution of the above system is given by the Lax-Hopf formula (see \cite[Theorem~5, Section~3.3.2]{evans10}\footnote{The proof in \cite{evans10} is given in the finite-dimensional case but it extends readily to any real Hilbert space.}):
	\[
	g\para{x,t} \eqdef  \inf\limits_{y\in\mc{H}}\brac{g_0(y) + tH^{*}\left(\frac{y-x}{t}\right)} .
	\]
	If $H(p)=\frac{1}{2}\norm{p}{}^2$, then $H^{*}(p)=\frac{1}{2}\norm{p}{}^2$ and we recover the result in \propref{prop:moreau}.
\end{enumerate}
\end{remark}

\paragraph{Regularity of differentiable functions}
In what follows, we introduce some definitions related with regularity of differentiable functions. They will provide useful upper-bounds and descent properties. Notice that the the notions and results of this part are independent from convexity.
\begin{definition}{($\omega$-smoothness)}\label{omegasmooth}
		Consider a function $\omega: \R_+ \to \R_+$ such that $\omega(0)=0$ and 
		\begin{equation}\label{def_xi}
		\xi\para{s} \eqdef  \int_{0}^1 \omega\para{st}dt
		\end{equation} is non-decreasing. A differentiable function $g:\mc{H}\to\R$ is said to belong to $C^{1,\omega}\para{\mc{H}}$ or to be $\omega$-smooth if the following inequality is satisfied for every $x,y\in\mc{H}$:
		\begin{equation*}
			\norm{\nabla g\para{x}-\nabla g\para{y}}{}\leq \omega\para{\norm{x-y}{}}.
		\end{equation*}
	\end{definition}
	\begin{lemma}{($\omega$-smooth Descent Lemma)}\label{prop:desclemma}
		Given a function $g\in C^{1,\omega}\para{\mc{H}}$ we have the following inequality: for every $x$ and $y$ in $\mc{H}$, 
		\begin{equation*}\label{omegadesc}
			g\para{y}-g\para{x} \leq \ip{\nabla g\para{x}, y-x}{} +\norm{y-x}{}\xi\para{\norm{y-x}{}},
		\end{equation*}
		where $\xi$ is defined in \eqref{def_xi}.
	\end{lemma}
	\begin{proof} We recall here the simple proof for completeness:
		\newq{
			g\para{y}-g\para{x}	&= \int_{0}^1 \frac{d}{dt}g\para{x+t\para{y-x}}dt\\
			&=\int_{0}^1\ip{\nabla g\para{x}, y-x}{}dt + \int_{0}^1\ip{\nabla g\para{x+t\para{y-x}}-\nabla g\para{x},y-x}{}dt\\
			&\leq \ip{\nabla g\para{x},y-x}{} + \norm{y-x}{}\int_{0}^1 \norm{\nabla g\para{x+t\para{y-x}}-\nabla g\para{x}}{}dt\\
			&\leq \ip{\nabla g\para{x}, y-x}{} + \norm{y-x}{}\int_{0}^1\omega\para{t\norm{y-x}{}}dt,
		}
	where in the first inequality we used Cauchy-Schwartz and in the second Definition \ref{omegasmooth}. We conclude using the definition of $\xi$. 
	\end{proof}
	For $L>0$ and $\omega \para{t} = Lt^\nu$, $\nu \in ]0,1]$, $C^{1,\omega}\para{\mc{H}}$ is the space of differentiable functions with H\"older continuous gradients, in which case $\xi\para{s} = Ls^\nu/(1+\nu)$ and the Descent Lemma reads
	\nnewq{\label{classicaldescent}
	g\para{y}-g\para{x} \leq \ip{\nabla g\para{x}, y-x}{} +\frac{L}{1+\nu}\norm{y-x}{}^{1+\nu} ,	
	}
	see e.g., \cite{Nesterov2015,nesterov2018}. When $\nu=1$, we have that $C^{1,\omega}\para{\mc{H}}$ is the class of differentiable functions with $L$-Lipschitz continuous gradient, and one recovers the classical Descent Lemma.
	
	Now, following \cite{BauschkeBolteTeboulle}, we introduce some notions that allow one to further generalize \eqref{classicaldescent}. Given a function $G:\mc{H}\to\R\cup\brac{\pinfty}$, differentiable on the open set $\C_0\subset\inte\left(\dom \left(G\right)\right)$, define the \emph{Bregman divergence} of $G$ as the function $D_G: \dom \left(G\right) \times \C_0 \to \R$,
	\begin{equation}\label{bregman}
		D_G(x,y)=G(x)-G(y)-\langle \nabla G(y),x-y\rangle.
	\end{equation} Then we have the following result. 
	\begin{lemma}{(Generalized Descent Lemma, \cite[Lemma 1]{BauschkeBolteTeboulle})}\label{desclemma}
		Let $G$ and $g$ be differentiable on $\C_0$, where $\C_0$ is an open subset of $\inte\left(\dom \left(G\right)\right)$. Assume that $G-g$ is convex on $\C_0$. Then, for every $x$ and $y$ in $\C_0$,
		$$g(y)\leq g(x)+\langle \nabla g(x), y-x\rangle +D_G(y,x).$$
	\end{lemma}
\begin{proof}
	For our purpose, we intentionally weakened the hypothesis needed in the original result of \cite[Lemma 1]{BauschkeBolteTeboulle}. We repeat their argument but show the result is still valid under our weaker assumption. Let $x$ and $y$ be in $\C_0$, where, by hypothesis, $\C_0$ is open and contained in $\inte\left(\dom \left(G\right)\right)$. As $G-g$ is convex and differentiable on $\C_0$, from the gradient inequality \eqref{subdiff} we have, for all $y\in \C_0$,
	$$\left(G-g\right)(y)\geq\left(G-g\right)(x)+\langle \nabla\left(G-g\right)(x), y-x\rangle.$$
	Rearranging the terms and using the definition of $D_G$ in \eqref{bregman}, we obtain the claim.
\end{proof}
	The previous lemma suggests the introduction of the following definition, which extends Definition \ref{omegasmooth}.
	\begin{definition}{($\left(G,\zeta\right)$-smoothness)}\label{def_smooth}
		Let $G:\mc{H}\to\R\cup\brac{\pinfty}$ and $\zeta: ]0,1] \to \R_+$. The pair $\left(g,\C\right)$, where $g:\mc{H}\to\R\cup\brac{\pinfty}$ and $\C\subset\dom(g)$, is said to be $\left(G,\zeta\right)$-smooth if there exists an open set $\C_0$ such that $\C\subset\C_0 \subset \inte\left(\dom \left(G\right)\right)$ and 
		\begin{enumerate}[label=(\roman*)]
			\item $G$ and $g$ are differentiable on $\C_0$;
			\item $G-g$ is convex on $\C_0$;
			\item it holds
			\begin{equation}\label{breg_const}
			K_{\left(G,\zeta,\C\right)}\ \eqdef \sup_{\substack{x,s\in \C;  \ \gamma\in]0,1] \\ z=x+\gamma\left(s-x\right)}} \frac{D_G(z,x)}{\zeta\left(\gamma\right)} \quad < \quad \pinfty.		
			\end{equation}
		\end{enumerate}
	\end{definition}
	
$K_{\left(G,\zeta,\C\right)}$ is a far-reaching generalization of the standard curvature constant widely used in the literature of conditional gradient.
	
\begin{remark}\label{remark}
Assume that $\left(g,\C\right)$ is $\left(G,\zeta\right)$-smooth. Using first Lemma \ref{desclemma} and then the definition in \eqref{breg_const}, we have the following descent property: for every $x,s\in\C$ and for every $\gamma\in]0,1]$,
	\begin{equation*}
		\begin{split}
			g\left(x+\gamma \left(s-x\right)\right) & \leq g(x)+\gamma\langle \nabla g(x), s-x\rangle +D_G(x+\gamma \left(s-x\right),x)\\
			& \leq g(x)+\gamma \langle \nabla g(x), s-x\rangle +K_{\left(G,\zeta,\C\right)} \zeta\left(\gamma\right).
		\end{split}
	\end{equation*}
Notice that, as in the previous definition, we do not require $\C$ to be convex. So, in general, the point $z=x+\gamma \left(s-x\right)$ may not lie in $\C$.
\end{remark}
	\begin{lemma}
	Suppose that the set $\C$ is bounded and denote by $\diam \eqdef  \sup_{x,y\in \C} \norm{x-y}{}$ its \emph{diameter}. Moreover, assume that the function $g$ is $\omega$-smooth on some open and convex subset $\C_0$ containing $\C$. Set $\zeta(\gamma)\eqdef \xi(\diam \gamma)$, where $\xi$ is given in \eqref{def_xi}. Then the pair $\left(g,\C\right)$ is $\left(g,\zeta\right)$-smooth with $K_{\left(g,\zeta,\C\right)}\leq \diam$.
	\end{lemma}
\begin{proof}
	With $G=g$ and $g$ being $\omega$-smooth on $\C_0$, both $G$ and $g$ are differentiable on $\C_0$ and $G-g\equiv0$ is convex on $\C_0$. Thus, all conditions required in \defref{def_smooth} hold true. It then remains to show \eqref{breg_const} with the bound $K_{\left(g,\zeta,\C\right)}\leq \diam$. First notice that, for every $x,s\in \C$ and for every $\gamma\in]0,1]$, the point $z=x+\gamma\left(s-x\right)$ belongs to $\C_0$. Indeed, $\C\subset\C_0$ and $\C_0$ is convex by hypothesis. In particular, as $g$ is $\omega$-smooth on $\C_0$, the Descent Lemma \ref{prop:desclemma} holds between the points $x$ and $z$. Then
	\begin{equation*}
	\begin{split}
		K_{\left(g,\zeta,\C\right)}&=\sup_{\substack{x,s\in \C;  \ \gamma\in]0,1] \\ z=x+\gamma\left(s-x\right)}} \frac{D_g(z,x)}{\zeta\left(\gamma\right)}\\
		&=\sup_{\substack{x,s\in \C;  \ \gamma\in]0,1] \\ z=x+\gamma\left(s-x\right)}} \frac{g(z)-g(x)-\langle \nabla g(x),z-x\rangle}{\xi(\diam \gamma)}\\
		&\leq\sup_{\substack{x,s\in \C;  \ \gamma\in]0,1] \\ z=x+\gamma\left(s-x\right)}} \frac{\norm{z-x}{}\xi\para{\norm{z-x}{}}}{\xi(\diam \gamma)}\\
		&=\sup_{\substack{x,s\in \C;  \ \gamma\in]0,1] }} \frac{\gamma\norm{s-x}{}\xi\para{\gamma\norm{s-x}{}}}{\xi(\diam \gamma)}\\
		&\leq\sup_{\gamma\in]0,1] } \frac{ \gamma \diam\xi\para{\diam\gamma }}{\xi(\diam \gamma)}=\diam.
	\end{split}		
	\end{equation*}
In the first inequality we used Lemma \ref{prop:desclemma}, while in the second we used that $\norm{s-x}{}\leq \diam$ (both $x$ and $s$ belong to $\C$, that is bounded by hypothesis) and the monotonicity of the function $\xi$ (see Definition \ref{omegasmooth}).
	\end{proof}

\paragraph{Indicator and support functions}
Given a subset $\C \subset \mc{H}$, we define its \emph{indicator function} as $\iota_{\C}(x)=0$ if $x$ belongs to $\C$ and $\iota_{\C}(x)=\pinfty$ otherwise. Recall that, if $\C$ is nonempty, closed, and convex, then $\iota_{\C}$ belongs to $\Gamma_0\para{\mc{H}}$. Remember also the definition of the \emph{support function} of $\C$, $\sigma_\C\eqdef  \iota_{\C}^{*}$. Equivalently, $\sigma_\C\para{x} \eqdef \sup\brac{\ip{z,x}{}: \ z\in \C}$. We denote by $\ri \left(\C\right)$ the \emph{relative interior} of the set $\C$ (in finite dimension, it is the interior for the topology relative to its affine full). We denote $\LinHull(C)$ as the subspace parallel to $\C$ which, in finite dimension, takes the form $\R(C-C)$. 

We have the following characterization of the support function from the relative interior in finite dimension. 
\begin{proposition}{(\cite[Lemma~1]{vaiter2015model})}\label{prop:supfun}
Let $\mc{H}$ be finite-dimensional and $\C\subset \mc{H}$ a nonempty, closed bounded and convex subset. If $0 \in \ri(\C)$, then $\sigma_{\C} \in \Gamma_0(\R^n)$ is sublinear, non-negative and finite-valued, and 
\newq{
\sigma_{\C}(x) = 0 \iff x \in (\LinHull(\C))^\bot.
}
\end{proposition}

\paragraph{Coercivity}
We recall that a function $g$ is \emph{coercive} if $\lim_{\norm{x}{}\to\pinfty} g\para{x}=\pinfty$ and that coercivity is equivalent to the boundedness of the sublevel-sets~\cite[Proposition~11.11]{BAUSCHCOMB}. We have the following result, that relates coercivity to properties of the Fenchel conjugate. 
\begin{proposition}{(\cite[Theorem 14.17]{BAUSCHCOMB})}\label{prop:coerc}
	Given $g$ in $\Gamma_0\para{\HH}$, $g^{*}$ is coercive if and only if $0\in \inte \left(\dom (g)\right)$.  
\end{proposition}

The \emph{recession function} (sometimes referred to as the horizon function) of $g$ at a given point $d\in\R^n$ is defined to be $g^{d,\infty}: \R^n \to \R\cup\left\{\pinfty\right\}$ such that, for every $x\in\R^n$,
	\newq{
		g^{d,\infty}\para{x} \eqdef  \lim\limits_{\alpha\to\infty}\frac{g\para{d+\alpha x}-g\para{d}}{\alpha}.
	}
Recall that, if $g$ is convex, the recession function is independent from the selection of the point $d\in\R^n$ and can be then simply denoted as $g^{\infty}$. In finite dimension, the following result relates coercivity to properties of the recession function. 
\begin{proposition}\label{prop:recfun}
Let $g\in\Gamma_0\para{\R^n}$ and $A:\R^m\to\R^n$ be a linear operator. Then,
\be[label=(\roman*)]
\item\label{eq:recfun1} 
$g\mbox{ coercive }\iff g^{\infty}\para{x}>0\quad\forall x\neq 0$.
\item\label{eq:recfun2} $g^{\infty} \equiv \sigma_{\dom\para{g^*}}$.
\item\label{eq:recfun3} $\para{g\circ A}^\infty\equiv g^\infty \circ A$.
\ee
In particular, we deduce that $g\circ A$ is coercive if and only if $\sigma_{\dom\para{g^*}}(Ax)>0$ for every $x\neq 0$.
\end{proposition}
\begin{proof}
The proofs can be found in \cite[Theorem 3.26]{rockafellar1998variational}, \cite[Theorem 11.5]{rockafellar1998variational} and \cite[Corollary 3.2]{Laghdir1999} respectively.
\end{proof}

\paragraph{Real sequences}
We close this section with some definitions and lemmas for real sequences that will be used to prove the convergence properties of the algorithm. We denote $\ell_+$ as the set of all sequences in $[0,+\infty[$. Given $p\in[1,\pinfty[$, $\ell^p$ is the space of real sequences $\seq{r_k}$ such that $\para{\sum\limits_{k=1}^\infty |r_k|^p}^{1/p} < \pinfty$. For $p=\pinfty$, we denote by $\ell^{\infty}$ the space of bounded sequences. Furthermore, we will use the notation $\ell^p_+ \eqdef \ell^p \cap \ell_+$. 
In the next, we recall some key results about real sequences.
\begin{lemma}{(\cite[Lemma 3.1]{combettesquasi})}\label{pre-Alb}
Consider three sequences $\seq{r_k} \in \ell_+$, $\seq{a_k} \in \ell_+$, and $\seq{z_k} \in \ell_+^1$, such that  
\[
r_{k+1} \leq r_k - a_k + z_k, \quad \forall k\in\N .
\] 
Then $\seq{r_k}$ is convergent and $\seq{a_k}\in\ell^1_+$.
\end{lemma}

\begin{lemma}{(\cite[Theorem 2]{Alber} and \cite[Proposition 2(ii)]{Alber})}\label{Alber1}
Consider two sequences $\seq{p_k} \in \ell_+$ and $\seq{w_k} \in \ell_+$ such that $\seq{p_k w_k}\in\ell^1_+$ and $\seq{p_k}\notin\ell^1$. Then the following holds:
\begin{enumerate}[label=(\roman*)]
\item there exists a subsequence $\subseq{w_{k_j}}$ such that
\[
w_{k_j}\leq P_{k_j}^{-1},
\]
where $P_{n} = \sum_{k=1}^n p_k$. In particular, $\liminf\limits_{k} w_k=0$.
\item If moreover there exists a constant $\alpha>0$ such that $w_k-w_{k+1} \leq \alpha p_k$ for every $k\in\N$, then 
\[
\lim\limits_{k} w_k=0 .
\]
\end{enumerate}
\end{lemma}

\begin{lemma}{}\label{Alber2}
Consider the sequences $\seq{r_k} \in \ell_+$, $\seq{p_k} \in \ell_+$, $\seq{w_k} \in \ell_+$, and $\seq{z_k} \in \ell_+$. Suppose that $\seq{z_k}\in\ell^1_+$, $\seq{p_k}\notin\ell^1$, and that, for some $\alpha>0$, the following inequalities are satisfied for every $k\in\N$:
\begin{equation}\label{eq:Alber2}
\begin{split}
& r_{k+1} \leq r_k - p_k w_k + z_k;\\
& w_k-w_{k+1}\leq \alpha p_k.
\end{split}
\end{equation}
Then,
\begin{enumerate}[label=(\roman*)]
\item $\seq{r_k}$ is convergent and $\seq{p_k w_k}\in\ell_+^1$. \label{alber2rk}
\item $\lim\limits_{k} w_k=0$.\label{alber2wk}
\item For every $k\in\N$, $\inf_{1 \leq i \leq k}w_{i}\leq \pa{r_0+E}/P_{k}$, where, again, $P_{n} = \sum_{k=1}^n p_k$ and $E=\sum_{k=1}^{\pinfty} z_k$.\label{alber2wkrateinf}
\item There exists a subsequence $\subseq{w_{k_j}}$ such that, for all $j\in\N$, $w_{k_j}\leq P_{k_j}^{-1}$.\label{alber2wkrate}
\end{enumerate}
\end{lemma}

\begin{proof}
\begin{enumerate}[label=(\roman*)]
\item See \lemref{pre-Alb}. 
\item Claim~\ref{alber2wk} follows by combining~\ref{alber2rk} and \lemref{Alber1}(ii).
\item Sum \eqref{eq:Alber2} using a telescoping property and summability of $\seq{z_k}$.
\item Claim~\ref{alber2wkrate} follows by combining~\ref{alber2rk} and \lemref{Alber1}(i).
\end{enumerate}
\end{proof}

Notice that the conclusions of \lemref{Alber2} remain true if non-negativity of the sequence $\seq{r_k}$ is replaced with the assumption that it is bounded from below by a trivial translation argument. Observe also that \lemref{Alber2} guarantees the convergence of the whole sequence to zero, but it gives a convergence rate only on a subsequence.

\section{Algorithm and assumptions}\label{sec:alg}

\begin{subsection}{Algorithm}
As described in the introduction, we combine penalization with the augmented Lagrangian approach to form the following functional
\nnewq{
\mc{J}_k\para{x,y,\mu} = f\para{x} + g\para{y} + h\para{x} + \ip{\mu, Ax-b}{}+\frac{\rho_k}{2}\norm{Ax-b}{}^2 + \frac{1}{2\beta_k}\norm{y-Tx}{}^2 ,
}
where $\mu$ is the dual multiplier, and $\rho_k$ and $\beta_k$ are non-negative parameters.
The steps of our scheme, then, are summarized in Algorithm~\ref{alg:CGAL}.
\begin{algorithm}[h]
    \SetAlgoLined
    \KwIn{$x_0\in \C = \dom\para{h}$; $\mu_0 \in \ran(A)$; $\seq{\gamma_k}$, $\seq{\beta_k}$, $\seq{\theta_k}, \seq{\rho_k} \in \ell_+$.}
    $k = 0$\\
    \Repeat{convergence}{
    $y_{k} =  \prox_{\beta_k g}\para{T x_k}$\\
    \vspace{0.25cm}
    $z_{k} = \nabla f(x_k)+T^*\left(Tx_k-y_k\right)/\beta_k+A^*\mu_k+\rho_kA^*\left(Ax_k-b\right)$\\
    \vspace{0.2cm}
    $s_{k} \in \Argmin_{s\in \HH_p}\brac{h\para{s} + \ip{z_k,s}{}}$\\
    \vspace{0.2cm}
    $x_{k+1} = x_k -\gamma_k\para{x_k-s_k}$\\
    \vspace{0.2cm}
    $\mu_{k+1} = \mu_k + \theta_k\para{Ax_{k+1}-b}$\\
    \vspace{0.2cm}
    $k \leftarrow k+1$\\
    \vspace{0.2cm}
    }
    \KwOut{$x_{k+1}$.}
\caption{Conditional Gradient with Augmented Lagrangian and Proximal-step (\cgal)}
\label{alg:CGAL}
\end{algorithm}
\ \\
For the interpretation of the algorithm, notice that the first step is equivalent to 
\[
\ens{y_k} = \Argmin_{y\in\HH_v}\mc{J}_k\para{x_k,y,\mu_k}.
\]
Now define the functional $\mc{E}_k\para{x,\mu} \eqdef  f\para{x} +g^{\beta_k}\para{T x}+\ip{\mu, Ax-b}{}+\frac{\rho_k}{2}\norm{Ax-b}{}^2.$ By convexity of the set $\C$ and the definition of $x_{k+1}$ as a convex combination of $x_k$ and $s_k$, the sequence $\seq{x_k}$ remains in $\C$ for all $k$, although the affine constraint $Ax_k=b$ might only be satisfied asymptotically. It is an augmented Lagrangian, where we do not consider the non-differentiable function $h$ and we replace $g$ by its Moreau envelope. Notice that
\begin{equation}\label{gradE}
\begin{split}
\nabla_x\Ek{x,\mu_k} & =\nabla f(x)+T^*[\nabla g^{\beta_k}](Tx)+A^*\mu_k+\rho_kA^*\left(Ax-b\right)\\
& = \nabla f(x)+\frac{1}{\beta_k}T^*\left(Tx-\prox_{\beta_k g}\para{T x}\right)+A^*\mu_k+\rho_kA^*\left(Ax-b\right).
\end{split}
\end{equation}
where in the second equality we used \ref{prop:moreau}\ref{moreauclaim2}. Then $z_k$ is just $\nabla_x \mc{E}_{k}\para{x_k,\mu_k}$ and the first three steps of the algorithm can be condensed in
\begin{equation}\label{sk}
s_{k} \in \Argmin\limits_{s\in \HH_p}\brac{h\para{s} + \ip{\nabla_x \mc{E}_{k}\para{x_k,\mu_k},s}{}}.
\end{equation}
Thus the primal variable update of each step of our algorithm boils down to conditional gradient applied to the function $\mc{E}_{k}\para{\cdot,\mu_k}$, where the next iterate is a convex combination between the previous one and the new direction $s_k$. A standard update of the Lagrange multiplier $\mu_k$ follows.
\end{subsection}

\subsection{Assumptions}
\subsubsection{Assumptions on the functions}
In order to help the reading, we recall in a compact form the following notation that we will use to refer to various functionals throughout the paper:
\nnewq{\label{notation}
	\Phi\para{x} &\eqdef  f\para{x} + g\para{Tx} + h\para{x};\\
	\Phi_k\para{x} &\eqdef  f\para{x} + g^{\beta_k}\para{Tx} + h\para{x} + \frac{\rho_k}{2}\norm{Ax-b}{}^2;\\
	\bar{\Phi}\para{x} &\eqdef \Phi\para{x}+\para{\rhosup/2}\norm{Ax-b}{}^2;\\
	\bar{\varphi}(\mu) &\eqdef \bar{\Phi}^{*}\left(-A^*\mu\right)+\scal{b}{\mu};\\
	\mc{L}\para{x,\mu} &\eqdef  f\para{x} + g\para{T x} + h\para{x} + \ip{\mu, Ax-b}{};\\
	\mc{L}_k\para{x,\mu} &\eqdef  f\para{x} + g^{\beta_k}\para{T x} + h\para{x} + \ip{\mu, Ax-b}{} + \frac{\rho_k}{2}\norm{Ax-b}{}^2;\\
	\mc{E}_k\para{x,\mu} &\eqdef  f\para{x} +g^{\beta_k}\para{T x}+\ip{\mu, Ax-b}{}+\frac{\rho_k}{2}\norm{Ax-b}{}^2 ,
}
where $\rhosup$ is defined in Assumption~\ref{ass:rhodec} to be $\rhosup = \sup\limits_k\rho_k$. 

In the list~\eqref{notation}, we can recognize $\Phi$ as the objective, $\Phi_k$ as the smoothed objective augmented with a quadratic penalization of the constraint, and $\mc{L}_k$ as a smoothed augmented Lagrangian. $\mc{L}$ denotes the classical Lagrangian. Recall that $\para{\xs,\mus}\in\HH_p\times\HH_d$ is a saddle-point for the Lagrangian $\mc{L}$ if for every $\para{x,\mu}\in\HH_p\times\HH_d$,
\begin{equation}\label{saddle_point}
\LL{\xs,\mu}\leq\LL{\xs,\mus}\leq\LL{x,\mus}.
\end{equation}
It is well-known from standard Lagrange duality, see e.g.~\cite[Proposition~19.19]{BAUSCHCOMB} or \cite[Theorem~3.68]{PEYPOU}, that the existence of a saddle point $\para{\xs,\mus}$ ensures strong duality, that $\xs$ solves \eqref{PProb} and $\mus$ solves the dual problem,
\begin{equation}\label{DProb}\tag{$\mathrsfs{D}$}
\min\limits_{\mu\in\HH_d} \pa{f + g \circ T + h}^*(-A^*\mu) + \ip{\mu,b}{} .
\end{equation}

The following assumptions on the problem will be used throughout the convergence analysis (for some results only a subset of these assumptions will be needed):\\
\begin{enumerate}[label=(A.\arabic*)]
	\item \label{ass:A1} $f,g\circ T$, and $h$ belong to $\Gamma_0\para{\mc{H}_p}$.
	\item \label{ass:f} The pair $\left(f,\C\right)$ is $\left(F,\zeta\right)$-smooth (see Definition \ref{def_smooth}), where we recall $\C\eqdef \dom\para{h}$.
	\item \label{ass:compact} $\C$ is weakly compact (and thus contained in a ball of radius $\Ccon>0$).
	\item \label{ass:interior}$T\C \subset \dom(\p g)$ and $\sup\limits_{x\in \C}\norm{\sbrac{\p g\para{Tx}}^0}{} < \infty$.
	\item \label{ass:lip} $h$ is Lipschitz continuous relative to its domain $\C$ with constant $L_h \geq 0$, i.e., $\allowbreak \forall (x,z) \in \C^2$, $|h(x)-h(z)| \leq L_h \norm{x-z}{}$. 
	\item \label{ass:existence} There exists a saddle-point $\para{\xs,\mus}\in\HH_p\times\HH_d$ for the Lagrangian $\mc{L}$.
	\item \label{ass:rangeA} $\ran(A)$ is closed.
	\item \label{ass:coercive} One of the following holds:
	\begin{enumerate}[label=(\alph*)]
	\item \label{ass:coerciveinfdim} $A^{-1}\para{b}\cap\inte\para{\dom\para{g\circ T}}\cap\inte\para{\C} \neq \emptyset$, where $A^{-1}\para{b}$ is the pre-image of $b$ under $A$.
	\item \label{ass:coercivefindim} $\HH_p$ and $\HH_d$ are finite dimensional and
	\nnewq{\label{eq:findimass}
	\begin{cases}
	A^{-1}\para{b} \cap \ri\para{\dom \para{g\circ T}} \cap \ri\para{\C}\neq\emptyset \\
	\qandq \\
	\ran\para{A^*}\cap \LinHull\para{\dom \para{g\circ T} \cap \C}^\bot = \brac{0} .
	\end{cases}
	}
	\end{enumerate}
\end{enumerate}

At this stage, a few remarks are in order.
\begin{remark}\label{rem:assfun}
{$~$}\\\vspace*{-0.5cm}
\begin{enumerate}[label=(\roman*)]
\item By Assumption~\ref{ass:A1}, $\C$ is also closed and convex. This together with Assumption \ref{ass:compact} entail, upon using \cite[Lemma~3.29 and Theorem~3.32]{BAUSCHCOMB}, that $\C$ is weakly compact.
\item Since the sequence of iterates $\seq{x_k}$ generated by Algorithm~\ref{alg:CGAL} is guaranteed to belong to $\C$ under \ref{ass:P1}, we have from \ref{ass:interior}
\nnewq{\label{MMM}
	\sup\limits_{k}\norm{\sbrac{\p g\para{Tx_k}}^0}{}\leq \gcon.
}
where $\gcon$ is a positive constant.
\item Assumption~\ref{ass:lip} will only be needed in the proof of convergence to optimality (Theorem~\ref{convergence}). It is not needed to show asymptotic feasibility (Theorem~\ref{thm:feas}).
\item \label{rem:assfunproper}
Assume that $A^{-1}(b) \cap \dom(g \circ T) \cap \C \neq \emptyset$, which entails that the set of minimizers of \eqref{PProb} is a non-empty convex closed bounded set under~\ref{ass:A1}-\ref{ass:compact}. Then there are various domain qualification conditions, e.g., one of the conditions in \cite[Proposition~15.24 and Fact~15.25]{BAUSCHCOMB}, that ensure the existence of a saddle-point for the Lagrangian $\mc{L}$ (see \cite[Theorem~19.1 and Proposition~9.19(v)]{BAUSCHCOMB}).
\item Observe that under the inclusion assumption of \lemref{lem:interior}, \ref{ass:coercive}\ref{ass:coerciveinfdim} is equivalent to $A^{-1}\para{b}\cap\inte\para{\C} \neq \emptyset$.
\item Assumption~\ref{ass:coercive} will be crucial to show that $\bar\varphi$ is coercive on $\ker(A^*)^\bot=\ran(A)$ (the last equality follows from \ref{ass:rangeA}), and hence boundedness of the dual multiplier sequence $\seq{\mu_k}$ provided by Algorithm~\ref{alg:CGAL} (see~\lemref{lem:coercive} and \lemref{bounded}). 
\end{enumerate}
\end{remark}

The uniform boundedness of the minimal norm selection on $\C$, as required in Assumption~\ref{ass:interior}, is important when we will invoke Proposition~\ref{prop:moreau}\ref{moreauclaim4} in our proofs to get meaningful estimates. The following result gives some sufficient conditions under which \ref{ass:interior} holds (in fact an even stronger claim).
\begin{lemma}{}\label{lem:interior}
Let $\C$ be a nonempty bounded subset of $\mc{H}_p$, $g\in\Gamma_0\para{\HH_v}$ and $T: \HH_p \to \HH_v$ be a bounded linear operator. Suppose that $T\C \subset \inte\para{\dom\para{g}}$. Then the assumption \ref{ass:interior} holds.
\end{lemma}
\begin{proof}
Since $g \in \Gamma_0\para{\mc{H}_p}$, it follows from~\cite[Proposition~16.21]{BAUSCHCOMB} that
\[
T \C \subset \inte\para{\dom\para{g}} \subset \dom(\p g). 
\]
Moreover, by~\cite[Corollary~8.30(ii) and Proposition~16.14]{BAUSCHCOMB}, we have that $\p g$ is locally weakly compact on $\inte\para{\dom\para{g}}$. In particular, as we assume that $\C$ is bounded, so is $T\C$, and since $T\C \subset \inte\para{\dom\para{g}}$, it means that for each $z \in T\C$ there exists an open neighborhood of $z$, denoted by $U_z$, such that $\p g\para{U_z}$ is bounded. Since $\para{U_z}_{z\in \C}$ is an open cover of $T\C$ and $T\C$ is bounded, there exists a finite subcover $\para{U_{z_k}}_{k=1}^n$. Then,
\newq{\bigcup\limits_{x\in \C}\p g\para{T x} \subset \bigcup\limits_{k=1}^n \p g\para{U_{z_{k}}}.
}
Since the right-hand-side is bounded (as it is a finite union of bounded sets), 
\newq{
\sup\limits_{x \in \C, \ u\in \p g \para{T x}}\norm{u}{} < \pinfty ,
}
whence the desired conclusion trivially follows.
\end{proof}

\subsubsection{Assumptions on the parameters}
We also use the following assumptions on the parameters of \algref{alg:CGAL} (recall the function $\zeta$ in \defref{def_smooth}):
\begin{enumerate}[label=(P.\arabic*)]
	\item \label{ass:P1} $\seq{\gamma_k}\subset]0,1]$ and the sequences $\seq{\zeta\left(\gamma_k\right)}, \seq{\gamma_k^2/\beta_k}$ and $\seq{\gamma_k\beta_k}$ belong to $\ell^1_+$.
	\item\label{notell1} $\seq{\gamma_k} \notin \ell^1$.
	\item \label{ass:beta} $\seq{\beta_k} \in \ell_+$ is non-increasing and converges to $0$.
	\item \label{ass:rhodec}\label{ass:rhobound} $\seq{\rho_k} \in \ell_+$ is non-decreasing with $0 < \rhoinf = \inf_k\rho_k \leq \sup_k \rho_k = \rhosup<\pinfty$.
	\item \label{ass:gkgkp1} For some positive constants $\gamconinf$ and $\gamcon$, $\gamconinf \leq \inf_{k} \para{\gamma_k / \gamma_{k+1}} \leq \sup_{k} \para{\gamma_k / \gamma_{k+1}} \leq \gamcon$.
	\item \label{ass:thetagamma} $\seq{\theta_k}$ satisfies $\theta_k =  \frac{\gamma_k}{c}$ for all $k\in\N$ for some $c>0$ such that $\frac{\gamcon}{c}-\frac{\rhoinf}{2}< 0$.
	\item \label{ass:cond} $\seq{\gamma_k}$ and $\seq{\rho_k}$ satisfy $\rho_{k+1}-\rho_k-\gamma_{k+1}\rho_{k+1}+\frac{2}{c}\gamma_k-\frac{\gamma_k^2}{c}\leq \gamma_{k+1}$ for all $k\in\N$ and for $c$ in \ref{ass:thetagamma}.
\end{enumerate}

\begin{remark}\label{rem:param}
{$~$}\\\vspace*{-0.5cm}
\begin{enumerate}[label=(\roman*)]
\item One can recognize that the update of the dual multiplier $\mu_k$ in \algref{alg:CGAL} has a flavour of  gradient ascent applied to the augmented dual with step-size $\theta_k$. However, unlike the standard method of multipliers with the augmented Lagrangian, Assumption~\ref{ass:thetagamma} requires $\theta_k$ to vanish in our setting. The underlying reason is that our update can be seen as an inexact dual ascent (i.e., exactness stems from the conditional gradient-based update on $x_k$ which is not a minimization of over $x$ of the augmented Lagrangian $\mc{L}_k$). Thus $\theta_k$ must annihilate this error asymptotically.
\item A sufficient condition for \ref{ass:cond} to hold consists of taking $\rho_k \equiv \rho > 0$ and $\gamma_{k+1}\geq\frac{2}{c(1+\rho)}\gamma_k$.
In particular, if $\seq{\gamma_k}$ satisfies \ref{ass:gkgkp1}, then, for \ref{ass:cond} to hold, it is sufficient to take $\rho_k \equiv \rho > 2\gamcon/c$ as supposed in \ref{ass:thetagamma}.\label{rem:paramrho}
\item The relevance of having $\rho_k$ vary is that it allows for more general and less stringent choice of the step-size $\gamma_k$. It is, however, possible (and easier in practice), to simply pick $\rho_k\equiv \rho$ for all $k\in\N$ as described above.
\end{enumerate}
\end{remark}

There is a large class of sequences that fulfill the requirements \ref{ass:P1}-\ref{ass:cond}. A typical one is as follows.
\begin{example}\label{ex:seq}
Take\footnote{Of course, one can add a scaling factor in the choice of the parameters which would allow for more practical flexibility. But this does not change anything to our discussion nor to the bahviour of the \cgal algorithm for $k$ large enough.}, for $k \in \N$,
\newq{
\rho_k \equiv \rho > 0, \gamma_k=\frac{(\log (k+2))^a}{(k+1)^{1-b}}, \beta_k = \frac{1}{(k+1)^{1-\delta}}, \qwithq\\
a \geq 0, 0 \leq 2b < \delta < 1, \delta < 1-b, \rho > 2^{2-b}/c, c > 0 . 
}
In this case, one can take the crude bounds $\gamconinf=(\log(2)/\log(3))^a$ and $\gamcon=2^{1-b}$, and choose $\rho > 2\gamcon/c$ as devised in \remref{rem:param}\ref{rem:paramrho}. In turn, \ref{ass:rhodec}-\ref{ass:cond} hold. In addition, suppose that $f$ has a $\nu$-H\"older continuous gradient (see~\eqref{classicaldescent}). Thus for \ref{ass:P1}-\ref{notell1} to hold, simple algebra shows that the allowable choice of $b$ is in $\left[0,\min\Ppa{1/3,\frac{\nu}{1+\nu}}\right[$.
\end{example}

\section{Convergence analysis}\label{sec:convan}
\subsection{Main results}

We state here our main results.

\begin{theorem}[Asymptotic feasibility]\label{thm:feas}
Suppose that Assumptions~\ref{ass:A1}-\ref{ass:interior} and \ref{ass:existence} hold. Consider the sequence of iterates $\seq{x_k}$ from Algorithm~\ref{alg:CGAL} with parameters satisfying Assumptions~\ref{ass:P1}-\ref{ass:thetagamma}. Then,
\begin{enumerate}[label=(\roman*)]
\item $Ax_{k}$ converges strongly to $b$ as $k\to\infty$, i.e., the sequence $\seq{x_k}$ is asymptotically feasible for \eqref{PProb} in the strong topology.
\item Pointwise rate:
\begin{equation}\label{pointratefeas}
	\begin{split}
	\inf_{0 \leq i \leq k} \norrm{Ax_{i}-b} = O\para{\frac{1}{\sqrt{\Gamma_{k}}}} \qandq \text{$\exists$ a subsequence $\subseq{x_{k_j}}$ s.t. for all $j\in\N$,} \norrm{Ax_{k_j}-b} \leq \frac{1}{\sqrt{\Gamma_{k_j}}},
	\end{split}
	\end{equation}
where, for all $k\in\N$, $\Gamma_k \eqdef \sum_{i=0}^{k} \gamma_i$.
\item Ergodic rate: for each $k\in\N$, let $\avx_k \eqdef \sum_{i=0}^{k} \gamma_i x_i/\Gamma_k$. Then
\begin{equation}\label{ergratefeas}
	\begin{split}
	\norrm{A\avx_k-b} = O\para{\frac{1}{\sqrt{\Gamma_{k}}}} .
	\end{split}
	\end{equation}
\end{enumerate}
\end{theorem}

Theorem~\ref{thm:feas} will be proved in \secref{sec:asfeas}.

\begin{theorem}[Convergence to optimality]\label{convergence}
Suppose that assumptions \ref{ass:A1}-\ref{ass:coercive} and \ref{ass:P1}-\ref{ass:cond} hold, with $\gamconinf \geq 1$. Let $\seq{x_k}$ be the sequence of primal iterates generated by \algref{alg:CGAL} and $(\xs,\mus)$ a saddle-point pair for the Lagrangian. Then, in addition to the results of Theorem~\ref{thm:feas}, the following holds
\begin{enumerate}[label=(\roman*)]
\item Convergence of the Lagrangian:
	\begin{equation}\label{limLL}
	\begin{aligned}
		\lim_{k\to\infty} \LL{x_{k},\mus} = \LL{\xs,\mus} .
	\end{aligned}
	\end{equation}
\item Every weak cluster point $\bar{x}$ of $\seq{x_k}$ is a solution of the primal problem \eqref{PProb}, and $\seq{\mu_k}$ converges weakly to $\bar{\mu}$ a solution of the dual problem \eqref{DProb}, i.e., $(\bar{x},\bar{\mu})$ is a saddle point of $\mc{L}$.
\item Pointwise rate: 
	\begin{equation}\label{pointrateopt}
	\begin{split}
		\inf_{0 \leq i \leq k} \LL{x_{i},\mus}-\LL{\xs,\mus} = O\para{\frac{1}{\Gamma_{k}}} \qandq \\
		\text{$\exists$ a subsequence $\subseq{x_{k_j}}$ s.t. for each $j\in\N$,} \LL{x_{k_j+1},\mus}-\LL{\xs,\mus} \leq \frac{1}{\Gamma_{k_j}} .
	\end{split}
	\end{equation}
\item Ergodic rate: for each $k\in\N$, let $\avx_k \eqdef \sum_{i=0}^{k} \gamma_ix_{i+1}/\Gamma_k$. Then
\begin{equation}\label{ergrateopt}
	\begin{split}
	\LL{\avx_{k},\mus}-\LL{\xs,\mus} = O\para{\frac{1}{\Gamma_{k}}} .
	\end{split}
	\end{equation}
\end{enumerate}
\end{theorem}

An important observation is that Theorem~\ref{convergence}, which will be proved in \secref{sec:opt}, actually shows that
\begin{equation*}
\lim_{k\to\infty} \left[\LL{x_{k},\mus}-\LL{\xs,\mus}+\frac{\rho_k}{2}\norrm{Ax_{k}-b}^2\right]=0,
\end{equation*}
and subsequentially, for each $j\in\N$,
\begin{equation}\label{rate}
\LL{x_{k_j},\mus}-\LL{\xs,\mus}+\frac{\rho_{k_j}}{2}\norrm{Ax_{k_j}-b}^2 \leq \frac{1}{\Gamma_{k_j}} .
\end{equation}
This means, in particular, that the pointwise rate for feasibility and optimality hold simulatenously for the same subsequence.
	
The following corollary is immediate.
\begin{corollary}
Under the assumptions of \thmref{convergence}, if the problem \eqref{PProb} admits a unique solution $\xs$, then the primal-dual pair sequence $\seq{x_k,\mu_k}$ converges weakly to a saddle point $(\xs,\mus)$. 
\end{corollary}

\begin{proof}
By uniqueness, it follows from Theorem~\ref{convergence}(ii) that $\seq{x_k}$ has exactly one weak sequential cluster point which is the solution to~\eqref{PProb}. Weak convergence of the sequence $\seq{x_k}$ then follows from \cite[Lemma~2.38]{BAUSCHCOMB}. 
\end{proof}

\begin{example}\label{ex:seqconv}
Suppose that the sequences of parameters are chosen according to Example~\ref{ex:seq}. Let the function $\sigma: t \in \R^+ \mapsto (\log(t+2))^a/(t+1)^{1-b}$. We obviously have $\sigma(k) = \gamma_k$ for $k \in \N$. Moreover, it is easy to see that $\exists k' \geq 0$ (depending on $a$ and $b$), such that $\sigma$ is decreasing for $t \geq k'$. Thus, $\forall k \geq k'$, we have
\[
\Gamma_k \geq \sum_{i=k'}^k \gamma_i \geq \int_{k'}^{k+1} \sigma(t) dt \geq \int_{k'+1}^{k+2} (\log(t))^a t^{b-1} dt = \int_{\log(k'+1)}^{\log(k+2)} t^{a} e^{b t} dt .
\]
It is easy to show, using integration by parts for the first case, that
\[
\Gamma_k^{-1} = 
\begin{cases}
o\para{\frac{1}{(k+2)^{b}}} & a=1, b > 0 , \\
O\para{\frac{1}{(k+2)^{b}}} & a=0, b > 0 , \\
O\para{\frac{1}{\log(k+2)}} & a=0, b=0 .
\end{cases}
\]
This result reveals that picking $a$ and $b$ as large as possible results in a faster convergence rate, with the proviso that $b$ satisfy some conditions for \ref{ass:P1}-\ref{ass:cond} to hold, see the discussion in Example~\ref{ex:seq} for the largest possible choice of $b$.
\end{example}

\subsection{Preparatory results}

The next result is a direct application of the Descent Lemma \ref{classicaldescent} and the generalized one in Lemma \ref{desclemma} to the specific case of Algorithm \ref{alg:CGAL}. It allows to obtain a descent property for the function $\Ek{\cdot,\mu_k}$ between the previous iterate $x_k$ and next one $x_{k+1}$.

\begin{lemma}\label{DescLemma}
	Suppose Assumptions \ref{ass:A1}, \ref{ass:f} and \ref{ass:P1} hold. For each $k\in\N$, define the quantity
	\begin{equation}\label{Lk}
	L_k\eqdef \frac{\norrm{T}^2}{\beta_k}+\norrm{A}^2\rho_k.
	\end{equation} 
	Then, for each $k\in\N$, we have the following inequality:
\begin{equation*}
	\begin{split}
		\Ek{x_{k+1},\mu_k} & \leq \Ek{x_k,\mu_k} + \scal{\nabla_x \Ek{x_k,\mu_k}}{x_{k+1}-x_k} +  K_{\left(F,\zeta,\C\right)}\zeta\left(\gamma_k\right) \\
				   & \qquad + \frac{L_k}{2}\norrm{x_{k+1}-x_k}^2. 
	\end{split}
\end{equation*}
	\end{lemma}

\begin{proof}
	Define for each $k\in\N$,
	\[
	\Etk{x,\mu}\eqdef g^{\beta_k}\para{T x}+\ip{\mu, Ax-b}{}+\frac{\rho_k}{2}\norm{Ax-b}{}^2,
	\]
	so that $\Ek{x,\mu}=f(x)+\Etk{x,\mu}$. Compute 
	\begin{equation*}
		\nabla_x\Etk{x,\mu}=T^*\nabla g^{\beta_k}(Tx)+A^*\mu+\rho_kA^*\left(Ax-b\right),
	\end{equation*}
	which is Lipschitz-continuous with constant $L_k=\frac{\norrm{T}^2}{\beta_k}+\norrm{A}^2\rho_k$ by virtue of \ref{ass:A1} and \propref{prop:moreau}\ref{moreauclaim2}.
	Then we can use the Descent Lemma \eqref{classicaldescent} with $\nu=1$ on $\Etk{\cdot,\mu_k}$ between the points $x_k$ and $x_{k+1}$, to obtain, for each $k\in\N$,
	\begin{equation}\label{e1}
	\Etk{x_{k+1},\mu_k}\leq\Etk{x_{k},\mu_k}+\langle \nabla \Etk{x_{k},\mu_k}, x_{k+1}-x_k\rangle + \frac{L_k}{2}\norm{x_{k+1}-x_k}{}^2.
	\end{equation}
	From Assumption \ref{ass:f}, Lemma \ref{desclemma} and Remark \ref{remark}, we have, for each $k\in\N$,
	\begin{equation*}
		\begin{split}
			f(x_{k+1}) & \leq f(x_k)+\langle \nabla f(x_k), x_{k+1}-x_k\rangle +D_F(x_{k+1},x_k)\\
			& \leq f(x_k)+\langle \nabla f(x_k), x_{k+1}-x_k\rangle +K_{\left(F,\zeta,\C\right)}\zeta\left(\gamma_k\right),
		\end{split}
	\end{equation*}
where we used that both $x_k$ and $s_k$ lie in $\C$, that $\gamma_{k}$ belongs to $]0,1]$ by \ref{ass:P1} and thus $x_{k+1}=x_k+\gamma_{k}\para{s_k-x_k} \in \C$.
	Summing \eqref{e1} with the latter and recalling that $\Ek{x,\mu_k}=f(x)+\Etk{x,\mu_k}$, we obtain the claim.
	\end{proof}
Again for the function $\Ek{\cdot,\mu_k}$, we also have a lower-bound, presented in the next lemma.
\begin{lemma}\label{lemma:lowerbound}
	Suppose Assumptions \ref{ass:A1} and \ref{ass:f} hold. Then, for all $k\in\N$, for all $x,x'\in\HH_p$ and for all $\mu\in\HH_d$,
	\begin{equation*}
		\begin{split}
			\Ek{x,\mu} & \geq \Ek{x',\mu} + \scal{\nabla_x \Ek{x',\mu}}{x-x'}+\frac{\rho_k}{2}\norrm{A(x-x')}^2.
		\end{split}
	\end{equation*}
\end{lemma}

\begin{proof}
	First, split the function $\Ek{\cdot,\mu}$ as $\Ek{x,\mu}=\Eko{x,\mu}+\frac{\rho_k}{2}\norrm{Ax-b}^2$ for an opportune definition of $\Eko{\cdot,\mu}$. For the first term, simply by convexity, we have
	\begin{equation}\label{a1}
	\begin{split}
	\Eko{x,\mu} & \geq \Eko{x',\mu} + \scal{\nabla_x \Eko{x',\mu}}{x-x'}.
	\end{split}
	\end{equation}
	Now use the strong convexity of the term $\left(\rho_k/2\right)\norrm{\cdot-b}^2$ between points $Ax$ and $Ax'$, to affirm that
	\begin{equation}\label{a2}
	\frac{\rho_k}{2}\norrm{Ax-b}^2 \geq \frac{\rho_k}{2}\norrm{Ax'-b}^2+\langle \nabla\left(\frac{\rho_k}{2}\norrm{\cdot-b}^2\right)\left(Ax'\right), Ax-Ax'\rangle +\frac{\rho_k}{2}\norrm{A(x-x')}^2.
	\end{equation}
	Compute
	\begin{equation*}
	\begin{split}
	\langle \nabla\left(\frac{\rho_k}{2}\norrm{\cdot-b}^2\right)\left(Ax'\right), Ax-Ax'\rangle & = \rho_k \langle A^*\left(Ax'-b\right), x-x'\rangle\\
	&=\langle \nabla\left(\frac{\rho_k}{2}\norrm{A\cdot-b}^2\right)\left(x'\right), x-x'\rangle.
	\end{split}
	\end{equation*}
	Summing \eqref{a1} and \eqref{a2} and invoking the gradient computation above, we obtain the claim.
\end{proof}

\begin{lemma}{}\label{lemma:LkBound}
Suppose that assumptions \ref{ass:A1}-\ref{ass:coercive} and \ref{ass:P1}-\ref{ass:cond} hold, with $\gamconinf \geq 1$. Let $\seq{x_k}$ be the sequence of primal iterates generated by \algref{alg:CGAL} and $\mus$ a solution of the dual problem \eqref{DProb}. Then we have the following estimate,
\newq{
\mc{L}\para{x_k,\mus} - \mc{L}\para{x_{k+1},\mus} \leq \gamma_kd_\C \para{M\norm{T}{} + D + L_h + \norm{A}{}\norm{\mus}{}}
}
\end{lemma}
\begin{proof}
First define $u_{k}\eqdef\left[\partial g(Tx_{k})\right]^0$ and recall that, by \ref{ass:interior} and its consequence in \eqref{MMM}, $\norrm{u_{k}} \leq \gcon$ for every $k \in \N$. Then,
	\newq{
	\LL{x_{k},\mus}-\LL{x_{k+1},\mus} &=\Phi(x_{k})-\Phi(x_{k+1})+\scal{\mus}{A\para{x_{k}-x_{k+1}}}\\
	&\leq \scal{u_{k}}{T(x_{k}-x_{k+1})}+\scal{\nabla f(x_{k})}{x_{k}-x_{k+1}}\\
	&\quad+L_h\norrm{x_{k}-x_{k+1}}+\norrm{\mus}\ \norrm{A}\ \norrm{x_{k}-x_{k+1}},
	}
where we used the subdifferential inequality \eqref{subdiff} on $g$, the gradient inequality on $f$, the $L_h$-Lipschitz continuity of $h$ relative to $\C$ (see~\ref{ass:lip}), and the Cauchy-Schwartz inequality on the scalar product. Since $x_{k+1}=x_{k}+\gamma_{k} \left(x_{k}-s_{k}\right)$, we obtain
	\begin{equation*}
	\begin{split}
	\LL{x_{k},\mus}-\LL{x_{k+1},\mus}	
	&\leq \gamma_{k} \Big(\scal{u_{k}}{T(x_{k}-s_{k})}+\scal{\nabla f(x_{k})}{x_{k}-s_{k}}+ L_h\norrm{x_{k}-s_{k}}\\
	&\quad+\norrm{\mus}\ \norrm{A}\ \norrm{x_{k}-s_{k}}\Big)\\
	&\leq \gamma_{k} \diam\para{\gcon\norrm{T}+D+L_h+\norrm{\mus}\ \norrm{A}},
	\end{split}
	\end{equation*}
	where we have denoted by $D$ the constant $D \eqdef \sup_{x\in\C}\norrm{\nabla f(x)} < +\infty$ (see ~\eqref{D}). 
\end{proof}

\begin{lemma}\label{lem:feasgammabound}
Suppose that assumptions \ref{ass:compact} and \ref{ass:rhodec} hold. Let $\seq{x_k}$ be the sequence of primal iterates generated by \algref{alg:CGAL}. Then we have the following estimate,
\newq{
\frac{\rho_k}{2}\norm{Ax_k-b}{}^2 - \frac{\rho_{k+1}}{2}\norm{Ax_{k+1}-b}{}^2 \leq \rhosup\diam\norrm{A}\left(\norrm{A}\Ccon+\norrm{b}\right)\gamma_k,
}
where $\Ccon$ is the radius of the ball containing $\C$ and $\rhosup = \sup_k\rho_k$.
\end{lemma}
\begin{proof}
By \ref{ass:rhodec} and convexity of the function $\frac{\rho_{k+1}}{2}\norrm{A\cdot-b}^2$, we have
	\begin{equation*}
	\begin{split}
	\frac{\rho_k}{2}\norrm{Ax_{k}-b}^2-\frac{\rho_{k+1}}{2}\norrm{Ax_{k+1}-b}^2 &\leq \frac{\rho_{k+1}}{2}\norrm{Ax_{k}-b}^2-\frac{\rho_{k+1}}{2}\norrm{Ax_{k+1}-b}^2\\ & \leq \scal{\nabla\left(\frac{\rho_{k+1}}{2}\norrm{A\cdot-b}^2\right)(x_{k})}{x_{k}-x_{k+1}}.
	\end{split}
	\end{equation*}
Now compute the gradient and use the definition of $x_{k+1}$, to obtain
\begin{equation*}
	\begin{split}
		\frac{\rho_k}{2}\norrm{Ax_{k}-b}^2-\frac{\rho_{k+1}}{2}\norrm{Ax_{k+1}-b}^2 &\leq\rho_{k+1}\gamma_{k}\scal{Ax_{k}-b}{A\left(x_{k}-s_{k}\right)}\\
		&\leq \rhosup\diam\norrm{A}\left(\norrm{A}\Ccon+\norrm{b}\right)\gamma_k.
	\end{split}
\end{equation*}
	In the last inequality, we used Cauchy-Schwartz inequality, triangle inequality, the fact that\\ $\norm{x_{k}-s_{k}}{}\leq \diam$, and assumptions \ref{ass:compact} and \ref{ass:rhodec} (respectively, $\sup_{x\in\C} \norrm{x}\leq \Ccon$ and $\rho_{k+1}\leq \rhosup$).
\end{proof}
\subsection{Asymptotic feasibility}\label{sec:asfeas}
We begin with an intermediary lemma establishing the main feasibility estimation and some summability results that will also be used in the proof of optimality.
\begin{lemma}\label{lem:feassum}
Suppose that Assumptions~\ref{ass:A1}-\ref{ass:interior} and \ref{ass:existence} hold. Consider the sequence of iterates $\seq{x_k}$ from Algorithm~\ref{alg:CGAL} with parameters satisfying Assumptions~\ref{ass:P1}-\ref{ass:thetagamma}. Define the two quantities $\Delta^p_k$ and $\Delta^d_k$ in the following way,
\newq{
\Delta^p_k \eqdef  \mc{L}_{k}\para{x_{k+1},\mu_k}-\minL_{k}\para{\mu_k},\quad \Delta^d_k \eqdef  \minL-\minL_{k}\para{\mu_k},
}
where we have denoted $\minL_{k}\para{\mu_k}\eqdef \min_{x}\mc{L}_k\para{x,\mu_k}$ and $\minL \eqdef  \mc{L}\para{\xs,\mus}$. Denote the sum $\Delta_{k} \eqdef  \Delta_k^p+\Delta_k^d$. Then we have the following estimation,
\newq{
\Delta_{k+1}\leq \Delta_k - \gamma_{k+1}\para{\frac{\gamconinf}{c}\norm{A\minx_{k+1} -b}{}^2 + \delta\norm{A\para{x_{k+1}-\minx_{k+1}}}{}^2} + \frac{L_{k+1}}{2}\gamma_{k+1}^2\diam^2\\
+K_{\para{F,\zeta,\C}}\zeta\para{\gamma_{k+1}}+\frac{\beta_{k}-\beta_{k+1}}{2}\gcon+\para{\frac{\rho_{k+1}-\rho_k}{2}}\norm{Ax_{k+1}-b}{}^2,
}
and, moreover,
\newq{
\seq{\gamma_{k}\norm{A\minx_{k}-b}{}^2} \in \ell^1_+, \quad \seq{\gamma_{k}\norm{A\para{x_{k}-\minx_{k}}}{}^2}\in\ell^1_+, \qandq\seq{\gamma_{k}\norm{Ax_{k}-b}{}^2}\in\ell^1_+.
}
\end{lemma}

\begin{proof}
First notice that the quantity $\Delta^p_k \geq 0$ and can be seen as a primal gap at iteration $k$ while $\Delta^d_k$ may be negative but is bounded from below by our assumptions. Indeed, in view of \ref{ass:A1}, \ref{ass:existence} and \remref{rem:assfun}\ref{rem:assfunproper}, $\minL_{k}\para{\mu_k}$ is bounded from above since
\newq{
\minL_{k}\para{\mu_k}	&\leq \mc{L}_k\para{\xs,\mu_k}\\
			&=f\para{\xs} + g^{\beta_k}\para{T\xs} + h\para{\xs} + \ip{\mu_k, A\xs-b}{} + \frac{\rho_k}{2}\norm{A\xs-b}{}^2\\
			&=f\para{\xs} + g^{\beta_k}\para{T\xs} + h\para{\xs} \\
			&\leq f\para{\xs} + g\para{T\xs} + h\para{\xs} < \pinfty ,
}
where we used \propref{prop:moreau}\ref{moreauclaim4} in the last inequality.

We denote a minimizer of $\mc{L}_k\para{x,\mu_k}$ by $\minx_k \in \Argmin\limits_{x\in\mc{H}_p}\mc{L}_k\para{x,\mu_k}$, which exists and belongs to $\C$ by \ref{ass:A1}-\ref{ass:compact}. Then, we have
\nnewq{\label{eq:dualgap}
\Delta^d_{k+1}-\Delta^d_k
&=\mc{L}_k\para{\minx_k,\mu_k} -\mc{L}_{k+1}\para{\minx_{k+1},\mu_{k+1}}.
}
Since $\minx_k$ is a minimizer of $\mc{L}_k\para{x,\mu_k}$ we have that $\mc{L}_k\para{\minx_k,\mu_k}\leq \mc{L}_k\para{\minx_{k+1},\mu_k}$ which leads to,
\newq{
\mc{L}_{k}\para{\minx_{k+1},\mu_k} 
&=\mc{L}_{k+1}\para{\minx_{k+1},\mu_k}+g^{\beta_k}\para{T\minx_{k+1}}- g^{\beta_{k+1}}\para{T\minx_{k+1}}+\tfrac{\rho_k-\rho_{k+1}}{2}\norm{A\minx_{k+1}-b}{}^2 \\
&\leq \mc{L}_{k+1}\para{\minx_{k+1},\mu_k},
}
where the last inequality comes from \propref{prop:moreau}\ref{moreauclaim4} and the assumptions \ref{ass:beta} and \ref{ass:rhodec}.
Combining this with \eqref{eq:dualgap},
\nnewq{\label{eq:dualgap2}
	\Delta^d_{k+1}-\Delta^d_k&\leq \mc{L}_{k+1}\para{\minx_{k+1},\mu_k}-\mc{L}_{k+1}\para{\minx_{k+1},\mu_{k+1}}\\
	&=\ip{\mu_k-\mu_{k+1},A\minx_{k+1}-b}{}\\
	&=-\theta_k\ip{Ax_{k+1}-b,A\minx_{k+1}-b}{},
}
where in the last equality we used the definition of $\mu_{k+1}$.
Meanwhile, for the primal gap we have
\newq{
\Delta_{k+1}^p-\Delta_k^p &= \para{\mc{L}_{k+1}\para{x_{k+2},\mu_{k+1}}-\mc{L}_k\para{x_{k+1},\mu_k}}+\para{\mc{L}_k\para{\minx_k,\mu_k}-\mc{L}_{k+1}\para{\minx_{k+1},\mu_{k+1}}}.
}
Note that
\[
\mc{L}_k\para{x_{k+1},\mu_k}=\mc{L}_k\para{x_{k+1},\mu_{k+1}}-\theta_k\norm{Ax_{k+1}-b}{}^2
\]
and estimate $\mc{L}_k\para{\minx_k,\mu_k}-\mc{L}_{k+1}\para{\minx_{k+1},\mu_{k+1}}$ as in \eqref{eq:dualgap2}, to get
\nnewqml{\label{eq:primalgap}
\Delta_{k+1}^p-\Delta_k^p	\leq  \mc{L}_{k+1}\para{x_{k+2},\mu_{k+1}}-\mc{L}_k\para{x_{k+1},\mu_{k+1}}+\theta_k\norm{Ax_{k+1}-b}{}^2 \\ 
-\theta_k\ip{Ax_{k+1}-b,A\minx_{k+1}-b}{}.
}
Using \eqref{eq:dualgap2} and \eqref{eq:primalgap}, we then have
\newqml{
\Delta_{k+1}-\Delta_k \leq \mc{L}_{k+1}\para{x_{k+2},\mu_{k+1}}-\mc{L}_k\para{x_{k+1},\mu_{k+1}}+\theta_k\norm{Ax_{k+1}-b}{}^2 \\
-2\theta_k\ip{Ax_{k+1}-b,A\minx_{k+1}-b}{}.
}
Note that 
\newq{
	\mc{L}_k\para{x_{k+1},\mu_{k+1}}=\mc{L}_{k+1}\para{x_{k+1},\mu_{k+1}}-\left[g^{\beta_{k+1}}-g^{\beta_k}\right]\para{T x_{k+1}}-\para{\frac{\rho_{k+1}-\rho_k}{2}}\norm{Ax_{k+1}-b}{}^2.
}
Then
\newqml{
	\Delta_{k+1}-\Delta_k \leq \mc{L}_{k+1}\para{x_{k+2},\mu_{k+1}}-\mc{L}_{k+1}\para{x_{k+1},\mu_{k+1}}+g^{\beta_{k+1}}\para{T x_{k+1}}-g^{\beta_k}\para{T x_{k+1}}\\
	 +\para{\frac{\rho_{k+1}-\rho_k}{2}}\norm{Ax_{k+1}-b}{}^2+\theta_k\norm{Ax_{k+1}-b}{}^2-2\theta_k\ip{Ax_{k+1}-b,A\minx_{k+1}-b}{}.
}
We denote by $\Te =\mc{L}_{k+1}\para{x_{k+2},\mu_{k+1}}-\mc{L}_{k+1}\para{x_{k+1},\mu_{k+1}}$ and the remaining part of the right-hand side by $\Tee$.	For the moment, we focus our attention on $\Te$. Recall that $\mc{L}_{k}\para{x,\mu_{k}}=\mc{E}_k\left(x,\mu_{k}\right)+h\left(x\right)$ and apply \lemref{DescLemma} between points $x_{k+2}$ and $x_{k+1}$, to get
\newqml{
\Te \leq h\para{x_{k+2}}-h\para{x_{k+1}}+\ip{\nabla_x \mc{E}_{k+1}\para{x_{k+1},\mu_{k+1}},x_{k+2}-x_{k+1}}{}\\
+K_{\para{F,\zeta,\C}}\zeta\para{\gamma_{k+1}}+\frac{L_{k+1}}{2}\norm{x_{k+2}-x_{k+1}}{}^2.
}
By \ref{ass:A1} we have that $h$ is convex and thus, since $x_{k+2}$ is a convex combination of $x_{k+1}$ and $s_{k+1}$, we get
\newqml{
\Te \leq -\gamma_{k+1}\para{h\para{x_{k+1}}-h\para{s_{k+1}}+\ip{\nabla_x\mc{E}_{k+1}\para{x_{k+1},\mu_{k+1}}, x_{k+1}-s_{k+1}}{}}\\
+\frac{L_{k+1}}{2}\norm{x_{k+2}-x_{k+1}}{}^2+K_{\para{F,\zeta,\C}}\zeta\para{\gamma_{k+1}}.
}
Applying the definition of $s_k$ as the minimizer of the linear minimization oracle and \lemref{lemma:lowerbound} at the points $\minx_{k+1}$, $x_{k+1}$, and $\mu_{k+1}$ gives,
\newq{
\Te &\leq -\gamma_{k+1}\para{h\para{x_{k+1}}-h\para{\minx_{k+1}}+\ip{\nabla_x\mc{E}_{k+1}\para{x_{k+1},\mu_{k+1}}, x_{k+1}-\minx_{k+1}}{}}\\
&\quad\quad +\frac{L_{k+1}}{2}\norm{x_{k+2}-x_{k+1}}{}^2+K_{\para{F,\zeta,\C}}\zeta\para{\gamma_{k+1}}\\
	&\leq -\gamma_{k+1}\Big(h\para{x_{k+1}}-h\para{\minx_{k+1}}+\mc{E}_{k+1}\para{x_{k+1},\mu_{k+1}}-\mc{E}_{k+1}\para{\minx_{k+1},\mu_{k+1}}\\
&\quad\quad +\frac{\rho_{k+1}}{2}\norm{A\para{x_{k+1}-\minx_{k+1}}}{}^2\Big)+\frac{L_{k+1}}{2}\norm{x_{k+2}-x_{k+1}}{}^2 +K_{\para{F,\zeta,\C}}\zeta\para{\gamma_{k+1}}\\
	&= -\gamma_{k+1}\para{\mc{L}_{k+1}\para{x_{k+1},\mu_{k+1}}-\mc{L}_{k+1}\para{\minx_{k+1},\mu_{k+1}}+\frac{\rho_{k+1}}{2}\norm{A\para{x_{k+1}-\minx_{k+1}}}{}^2} \\
&\quad\quad +\frac{L_{k+1}}{2}\norm{x_{k+2}-x_{k+1}}{}^2+K_{\para{F,\zeta,\C}}\zeta\para{\gamma_{k+1}} \\
	&\leq -\frac{\gamma_{k+1}\rho_{k+1}}{2}\norm{A\para{x_{k+1}-\minx_{k+1}}}{}^2+\frac{L_{k+1}}{2}\norm{x_{k+2}-x_{k+1}}{}^2+K_{\para{F,\zeta,\C}}\zeta\para{\gamma_{k+1}} ,
}
where we used that $\minx_{k+1}$ is a minimizer of $\mc{L}_{k+1}\para{\cdot,\mu_{k+1}}$ in the last inequality.
Now, combining $\Te$ and $\Tee$ and using the Pythagoras identity we have
\nnewqml{\label{eq:aux}
\Delta_{k+1}-\Delta_k \leq -\theta_{k}\norm{A\minx_{k+1}-b}{}^2+\para{\theta_{k}-\gamma_{k+1}\frac{\rho_{k+1}}{2}}\norm{A\para{x_{k+1}-\minx_{k+1}}}{}^2\\
+\frac{L_{k+1}}{2}\norm{x_{k+2}-x_{k+1}}{}^2+K_{\para{F,\zeta,\C}}\zeta\para{\gamma_{k+1}}+\left[g^{\beta_{k+1}}-g^{\beta_k}\right]\para{T x_{k+1}} \\ +\frac{\rho_{k+1}-\rho_k}{2}\norm{Ax_{k+1}-b}{}^2.
}

Under \ref{ass:thetagamma} we have $\theta_k=\frac{\gamma_{k}}{c}$ for some $c>0$ such that
\newq{
	\exists \delta>0,\quad \frac{\gamcon}{c} -\frac{\rhoinf}{2}=-\delta < 0,
}
where $\gamcon$ is the constant such that $\gamma_{k}\leq \gamcon \gamma_{k+1}$ (see Assumption \ref{ass:gkgkp1}).
Then, using \ref{ass:gkgkp1} and the above inequality,
\nnewq{\label{eq:aux1}
	\theta_{k}-\gamma_{k+1}\frac{\rho_{k+1}}{2}\leq \para{\frac{\gamcon}{c}-\frac{\rho_{k+1}}{2}}\gamma_{k+1}\leq\para{\frac{\gamcon}{c}-\frac{\rhoinf}{2}}\gamma_{k+1}=-\delta\gamma_{k+1} \tandt \theta_{k} \geq \tfrac{\gamconinf \gamma_{k+1}}{c} .
}
Now use the fact that $x_{k+2} = x_{k+1}+\gamma_{k+1}\para{s_{k+1}-x_{k+1}}$ to estimate
\begin{equation}\label{eq:aux2}
\norm{x_{k+2}-x_{k+1}}{}^2\leq \gamma_{k+1}^2\diam^2.
\end{equation}
Moreover, by the two assumptions~\ref{ass:beta}, \ref{ass:interior} and \propref{prop:moreau}\ref{moreauclaim4}, \eqref{MMM} holds with a constant $\gcon > 0$ , and thus with \propref{prop:moreau}\ref{moreauclaim3} we obtain
\nnewq{\label{eq:aux3}
	\left[g^{\beta_{k+1}}-g^{\beta_k}\right]\para{T x_{k+1}}\leq\frac{\beta_{k}-\beta_{k+1}}{2}\norm{\mat{\p g\para{T x_{k+1}}}^0}{}^2\leq \frac{\beta_{k}-\beta_{k+1}}{2}\gcon.
} 
Plugging \eqref{eq:aux1}, \eqref{eq:aux2} and \eqref{eq:aux3} into \eqref{eq:aux}, we get
\nnewq{\label{eq:estfeas}
	\Delta_{k+1}-\Delta_k	&\leq -\frac{\gamconinf}{c}\gamma_{k+1}\norm{A\minx_{k+1}-b}{}^2-\delta\gamma_{k+1}\norm{A\para{x_{k+1}-\minx_{k+1}}}{}^2+\frac{L_{k+1}}{2}\gamma_{k+1}^2\diam^2\\
	&\quad\quad +K_{\para{F,\zeta,\C}}\zeta\para{\gamma_{k+1}}+\frac{\beta_{k}-\beta_{k+1}}{2}\gcon+\para{\frac{\rho_{k+1}-\rho_k}{2}}\norm{Ax_{k+1}-b}{}^2.
}
Because of the assumptions \ref{ass:P1} and \ref{ass:rhobound}, and in view of the definition of $L_{k}$ in \eqref{Lk}, we have the following,
\newq{
\frac{L_{k}}{2}\gamma_{k}^2\diam^2 &= \frac{1}{2}\para{\frac{\norm{T}{}^2}{\beta_k}+\norm{A}{}^2\rho_k}\gamma_{k}^2\diam^2 \in\ell^1_+.
}
For the telescopic terms from the right hand side of \eqref{eq:estfeas} we have
\newq{
\frac{\beta_{k}-\beta_{k+1}}{2} \in\ell^1_+ \tandt \para{\frac{\rho_{k+1}-\rho_k}{2}}\norm{Ax_{k+1}-b}{}^2\leq \para{\rho_{k+1}-\rho_k}\para{\norm{A}{}^2 \Ccon^2+\norm{b}{}^2}\in\ell^1_+,
}
where $\Ccon$ is the constant arising from \ref{ass:compact}.
Under \ref{ass:P1} we also have that
\newq{
K_{\para{F,\zeta,\C}}\zeta\para{\gamma_{k+1}}\in\ell^1_+.
}

Using the notation of Lemma~\ref{Alber2}, we set
\newq{
r_k &= \Delta_k, \quad p_k = \gamma_{k+1}, \quad w_k = \para{\frac{\gamconinf}{c}\norm{A\minx_{k+1} -b}{}^2 + \delta\norm{A\para{x_{k+1}-\minx_{k+1}}}{}^2}, \\
z_k &= \frac{L_{k+1}}{2}\gamma_{k+1}^2\diam^2+K_{\para{F,\zeta,\C}}\zeta\para{\gamma_{k+1}}+\frac{\beta_{k}-\beta_{k+1}}{2}\gcon+\para{\frac{\rho_{k+1}-\rho_k}{2}}\norm{Ax_{k+1}-b}{}^2 .
}
We have shown above that 
\[
r_{k+1} \leq r_{k} - p_{k}w_{k} + z_k ,
\]
where $\seq{z_k} \in \ell_+^1$, and $r_k$ is bounded from below. We then deduce using Lemma~\ref{Alber2}\ref{alber2rk} that $\seq{r_k}$ is convergent and 
\nnewq{\label{eq:starl1}
\seq{\gamma_{k}\norm{A\minx_{k}-b}{}^2} \in \ell^1_+, \quad \seq{\gamma_{k}\norm{A\para{x_{k}-\minx_{k}}}{}^2}\in\ell^1_+.
}
Consequently, 
\nnewq{\label{eq:feasxkl1}
\seq{\gamma_{k}\norm{Ax_{k}-b}{}^2}\in\ell^1_+,
}
since, by Jensen's inequality,
\newq{
\sum\limits_{k=1}^\infty \gamma_{k}\norm{Ax_{k}-b}{}^2 &\leq 2\sum\limits_{k=1}^\infty \gamma_{k}\para{\norm{A\para{x_{k}-\minx_{k}}}{}^2+\norm{A\minx_{k}-b}{}^2} < \pinfty.
}
\end{proof}
We are now ready to prove Theorem~\ref{thm:feas}, i.e., to show that the sequence of iterates $\seq{x_k}$ is asymptotically feasible. 
\begin{proof}
\begin{enumerate}[label=(\roman*)]
\item By \lemref{lem:feasgammabound} with $\rho_k\equiv\rho_{k+1}\equiv2$, we have
\newq{
\norm{Ax_{k}-b}{}^2-\norm{Ax_{k+1}-b}{}^2 &\leq 2\gamma_k \diam\norm{A}{}\para{\norm{A}{}\Ccon + \norm{b}{}}.
}
Using this together with \lemref{lem:feassum} and Assumption~\ref{notell1}, we are in position to apply \lemref{Alber2}\ref{alber2wk} to conclude that $\lim_{k\to\infty}\norm{Ax_k-b}{} = 0$.

\item The rates in \eqref{pointratefeas} follow respectively from \lemref{Alber2}\ref{alber2wkrateinf} and \lemref{Alber2}\ref{alber2wkrate}.

\item We have, by Jensen's inequality and \lemref{lem:feassum}, that
\[
\norm{A\avx_k - b}{}^2 \leq \frac{1}{\Gamma_k}\sum_{i=0}^k \gamma_i \norm{Ax_i - b}{}^2 \leq \frac{1}{\Gamma_k}\sum_{i=0}^{\pinfty} \gamma_i \norm{Ax_i - b}{}^2 = O\para{\frac{1}{\Gamma_{k}}} .
\]
\end{enumerate}
\end{proof}
\subsection{Dual multiplier boundedness}
In this part we provide a lemma that shows the sequence of dual variables $\seq{\mu_k}$ generated by \algref{alg:CGAL} is bounded.

We start by studying coercivity of $\bar{\varphi}$.
\begin{lemma}\label{lem:coercive}
Suppose that Assumptions \ref{ass:A1}-\ref{ass:compact} and \ref{ass:existence}-\ref{ass:coercive} hold. Then $\bar\varphi$ is coercive on $\ran\para{A}$.
\end{lemma}
\begin{proof}
From~\eqref{notation}, we have, for any $c \in A^{-1}(b)$, that
\[
\bar{\varphi}(\mu)  = \para{\bar\Phi^* + \ip{-c, \cdot}{}}\para{-A^*\mu} .
\]
Moreover, Assumptions~\ref{ass:A1} and \ref{ass:rangeA} entail that $\bar{\Phi} \in \Gamma_0(\mc{H}_p)$. 
We now consider separately the two assumptions.
\begin{enumerate}[label=(\alph*)]
\item Case of~\ref{ass:coercive}\ref{ass:coerciveinfdim}:
If follows from the Fenchel-Moreau theorem (\cite[Theorem~13.32]{BAUSCHCOMB}) that 
\[
\para{\bar\Phi^* - \ip{c,\cdot}{}}^*=\bar\Phi^{**}\para{\cdot+c}=\bar\Phi\para{\cdot+c} .
\] 
Using this, together with Proposition~\ref{prop:coerc} and \ref{ass:f}, we can assert that $\bar\Phi^* - \ip{c,\cdot}{}$ is coercive if and only if 
\newq{
0 \in \inte\para{\dom\para{\bar\Phi\para{\cdot+c}}} 
= \inte\para{\dom\para{\Phi}} - c 
&= \inte\para{\dom\para{g \circ T} \cap \C} - c \\
&= \inte\para{\dom\para{g \circ T}} \cap \inte\para{\C} - c .
}
But this is precisely what \ref{ass:coercive}\ref{ass:coerciveinfdim} guarantees. In turn, using \cite[Proposition~14.15]{BAUSCHCOMB}, \ref{ass:coercive}\ref{ass:coerciveinfdim} is equivalent to 
\[
\exists (a > 0, \beta \in \R), \quad \bar\Phi^* - \ip{c,\cdot}{} \geq a\norm{\cdot}{} + \beta .
\]
Using standard results on linear operators in Hilbert spaces~\cite[Facts~2.18 and~2.19]{BAUSCHCOMB}, we have 
\[
\text{\ref{ass:rangeA}} \iff (\exists \alpha > 0), (\forall \mu \in \ran(A)), \quad \norm{A^*\mu}{} \geq \alpha \norm{\mu} .
\]
Combining the last two inequalities, we deduce that under \ref{ass:coercive}\ref{ass:coerciveinfdim},
\[
\exists (a > 0, \alpha > 0, \beta \in \R), (\forall \mu \in \ran(A)), \quad \bar{\varphi}(\mu) \geq a\norm{A^*\mu}{} + \beta \geq a\alpha\norm{\mu}{} + \beta ,
\]
which in turn is equivalent to coercivity of $\bar{\varphi}$ on $\ran(A)$ by \cite[Proposition~14.15]{BAUSCHCOMB}.

\item Case of~\ref{ass:coercive}\ref{ass:coercivefindim}: 
Since $\HH_d$ is finite dimensional, We have, $\forall u\in\HH_d$,
	\newq{
		\bar\varphi^\infty\para{u} 		&= \para{\para{\bar\Phi^* + \ip{-c, \cdot}{}}\circ\para{-A^*}}^\infty\para{u}\\
\text{\small{(\propref{prop:recfun}\ref{eq:recfun3})}}\quad	&= \para{\bar\Phi^* + \ip{-c, \cdot}{}}^\infty\para{-A^*u}\\
\text{\small{(\propref{prop:recfun}\ref{eq:recfun2})}}\quad	&= \sigma_{\dom\para{\bar\Phi^* + \ip{-c, \cdot}{}}^*}\para{-A^*u}\\
							&= \sigma_{\dom\para{\bar\Phi\para{\cdot+c}}}\para{-A^*u} \\
							&= \sigma_{\dom\para{\bar\Phi}-c}\para{-A^*u} \\
\text{\small{(by \ref{ass:f})}}\quad			&= \sigma_{\dom\para{g \circ T} \cap \C - c}\para{-A^*u} .
	}
	Notice that, by Assumption \ref{ass:interior}, we have $\dom\para{g \circ T} \cap \C=\C$. Thus, using \propref{prop:recfun}\ref{eq:recfun1}, we have the following chain of equivalences
	\newq{
	\text{$\bar\varphi$ is coercive on $\ran\para{A}$} 
	&\iff \bar\varphi^\infty\para{u}>0, \quad \forall u\in\ran\para{A}\setminus\brac{0} \\
	&\iff \sigma_{\C - c}(-A^*u) > 0, \quad \forall u\in\ran\para{A}\setminus\brac{0} .
	}
	For this to hold, and since $\ran\para{A} = \ker \para{A^*}^\bot$, a sufficient condition is that
	\nnewq{\label{eq:sigDpos}
	\sigma_{\C - c}(x) > 0,  \quad \forall x\in\ran\para{A^*}\setminus\brac{0} .
	}
	It remains to check that the latter condition holds under~\ref{ass:coercive}\ref{ass:coercivefindim}. First, observe that $\C$ is a nonempty bounded convex set thanks to \ref{ass:A1} and \ref{ass:compact}. The first condition in \ref{ass:coercive}\ref{ass:coercivefindim} is equivalent to $0 \in \ri(\C - c)$ for some $c \in A^{-1}(b)$. It then follows from \propref{prop:supfun} that 
	\newq{
	\sigma_{\C - c}(x) > 0, \forall x \not\in \LinHull(\C - c)^\bot = \LinHull(\C)^\bot ,
	}
	which then implies \eqref{eq:sigDpos} thanks to the second condition in \ref{ass:coercive}\ref{ass:coercivefindim}.
\end{enumerate}
\end{proof}


\begin{lemma}\label{bounded}
Suppose that assumptions  \ref{ass:A1}-\ref{ass:compact} and \ref{ass:existence}-\ref{ass:coercive} and \ref{ass:P1}-\ref{ass:thetagamma} hold. Then the sequence of dual iterates $\seq{\mu_k}$ generated by \algref{alg:CGAL} is bounded.
\end{lemma}


\begin{proof}
	Using the notation in \eqref{notation}, the primal problem:
	\begin{equation*}
	\begin{split}
	& \min_{x\in\HH_p} \left\{\Phi(x): \ \ Ax=b \right\}=\min_{x\in\HH_p} \sup_{\mu\in\HH_d}  \ \LL{x,\mu}  ,
	\end{split}
	\end{equation*}
	is obviously equivalent to
	\begin{equation*}
	\begin{split}
	& \min_{x\in\HH_p} \left\{\Phi(x)+\frac{\rho_k}{2}\norrm{Ax-b}^2: \ \ Ax=b \right\}= \min_{x\in\HH_p} \sup_{\mu\in\HH_d} \left\{ \LL{x,\mu} +\frac{\rho_k}{2}\norrm{Ax-b}^2 \right\}.
	\end{split}
	\end{equation*}
	We associate to the previous the following regularized primal problem:
	\begin{equation*}
	\begin{split}
	&\min_{x\in\HH_p} \left\{\Phi_k(x): \ \ Ax=b \right\}= \min_{x\in\HH_p} \sup_{\mu\in\HH_d} \ \Lk{x,\mu} 
	\end{split}
	\end{equation*}
	and its Lagrangian dual, namely:
	\begin{equation*}
	\begin{split}
	& \sup_{\mu\in\HH_d} \inf_{x\in\HH_p}  \  \Lk{x,\mu}=-\inf_{\mu\in\HH_d} \sup_{x\in\HH_p}  \  -\Lk{x,\mu}.
	\end{split}
	\end{equation*}
	Now consider the dual function in the latter, namely $\varphi_k(\mu) \eqdef -\inf_{x\in\HH_p}  \ \Lk{x,\mu}$. Observe that the minimum is actually attained owing to \ref{ass:A1} and \ref{ass:compact}.
Now we claim that $\varphi_k$ is continuously differentiable with $L_{\nabla\varphi_k}$-Lipschitz gradient, and $1/\rhoinf$ (see~\ref{ass:rhobound}) is an upper-bound for $\seq{L_{\nabla\varphi_k}}$. In order to show it, introduce the notation
	\begin{eqnarray*}
		F_k(x) &\eqdef & f(x)+g^{\beta_k}(Tx)+h(x);\\
		G_k(v) &\eqdef & \frac{\rho_k}{2}\norrm{v-b}^2.
	\end{eqnarray*}
By definition, we have
	\begin{equation}\label{eq:rockaf}
		\begin{split}
			\varphi_{k}(\mu)
			& = -\min_{x\in\HH_p} \left\{f(x)+g^{\beta_k}(Tx)+h(x)+\scal{\mu}{Ax-b}+\frac{\rho_k}{2}\norrm{Ax-b}^2\right\}\\
			& = -\min_{x\in\HH_p} \left\{F_k(x)+\scal{A^*\mu}{x}+G_k(Ax)\right\} + \scal{\mu}{b}.
		\end{split}
	\end{equation}
Using Fenchel-Rockafellar duality and strong duality, which holds by \ref{ass:rhobound} and continuity of $G_k$ (see, for instance, \cite[Theorem~3.51]{PEYPOU}), we have the following equality,
\newq{
\min_{x\in\HH_p}\brac{F_k\para{x} + \ip{A^*\mu, x}{} + G_k\para{Ax}} 	&= -\min\limits_{v\in\HH_d}\brac{ \para{F_k\para{\cdot} + \ip{A^*\mu, \cdot}{}}^*\para{-A^*v} + G_k^*\para{v}}\\
				&= - \min\limits_{v\in\HH_d}\brac{F_k^*\para{-A^*v-A^*\mu} + G_k^*\para{v}}
}
where we have used the fact that the conjugate of a linear perturbation is the translation of the conjugate in the last line. Substituting the above into \eqref{eq:rockaf} we find
	\begin{equation*}
		\begin{split}
		\varphi_{k}(\mu)	& = \min_{v\in\HH_d}  \left\{F_k^{*}(-A^*(v+\mu))+\frac{1}{2\rho_k}\norrm{v}^2 + \scal{v}{b} \right\} + \scal{\mu}{b} \\
			& = \min_{v\in\HH_d}  \left\{F_k^{*}(-A^*(v+\mu))+\frac{1}{2\rho_k}\norrm{v+\rho_k b}^2 \right\} + \scal{\mu}{b} - \frac{\rho_k}{2}\norrm{b}^2
		\end{split}
	\end{equation*}
 Moreover, from the primal-dual extremality relationships~\cite[Theorem~3.51(i)]{PEYPOU}, we have
	\begin{equation}\label{eq:extrel}
	-\minv = \nabla G_k(A\minx) = \rho_k\para{A\minx-b} ,
	\end{equation}
	where $\minx$ is a minimizer (which exists and belongs to $\C$) of the primal objective $\Lk{\cdot,\mu}$ and $\minv$ is the unique minimizer to the associated dual objective. Now, using the change of variable $u=v+\mu$, we get
	\begin{equation*}
		\begin{split}
			\varphi_{k}(\mu) & = \inf_{u\in\HH_d}  \left\{F_k^{*}(-A^*u)+\frac{1}{2\rho_k}\norrm{u-\mu+\rho_k b}^2 \right\} + \scal{\mu}{b} - \frac{\rho_k}{2}\norrm{b}^2\\
			& = \sbrac{F_k^{*}\circ(-A^*)}^{\rho_k}\para{\mu-\rho_k b} + \scal{\mu}{b} - \frac{\rho_k}{2}\norrm{b}^2,
		\end{split}
	\end{equation*}
	where the notation $\sbrac{\cdot}^{\rho_k}$ denotes the Moreau envelope with parameter $\rho_k$ as defined in \eqref{moreau_env}. It follows from \propref{prop:moreau}\ref{moreauclimconv} and \ref{moreauclaim2}, that $\varphi_{k}$ is convex, real-valued and its gradient, given by
	\nnewq{\label{eq:gradphik}
	\nabla \varphi_k(\mu) 	&= \rho_k^{-1}\para{\mu-\rho_k b - \minu} + b = \rho_k^{-1}\para{\mu - \minu} , \qwhereq \minu = \prox_{\rho_k F_k^{*}\circ(-A^*)}(\mu-\rho_k b),
	}
is $1/\rho_k$-Lipschitz continuous since the gradient of a Moreau envelope with parameter $\rho_k$ is $1/\rho_k$-Lipschitz continuous (see \propref{prop:moreau}\ref{moreauclaim2}). As $\rho_k$ is non-decreasing, $1/\rho_k\leq1/\rhoinf$ and the sequence of functions $\seq{\nabla\varphi_k}$ is uniformly Lipschitz-continuous with constant $1/\rhoinf$. In addition, combining~\eqref{eq:extrel} and~\eqref{eq:gradphik}, and recalling the change of variable $\minu=\minv+\mu$, we get that
	\nnewq{\label{eq:gradphikxk}
	\nabla \varphi_k(\mu) = \rho_k^{-1}(\mu-\minu) = -\rho_k^{-1}\minv = A\minx-b .
	}
	As in \lemref{lem:feassum}, we are going to denote $\minx_k$ a minimizer of $\Lk{x,\mu_k}$. Then, from the Descent Lemma (see Proposition \ref{prop:desclemma} and inequality \eqref{classicaldescent}), we have
	\begin{equation*}
	\begin{split}
	\varphi_k(\mu_{k+1}) & \leq \varphi_k(\mu_k)+\scal{\nabla \varphi_k(\mu_k)}{\mu_{k+1}-\mu_k}+\frac{1}{2\rhoinf}\norrm{\mu_{k+1}-\mu_k}^2.
	\end{split}
	\end{equation*}
	Now substitute in the right-hand-side the expression $\nabla \varphi_k(\mu_k)=A\minx_k-b$ in~\eqref{eq:gradphikxk} and the update $\mu_{k+1}=\mu_k+\theta_{k}\left(Ax_{k+1}-b\right)$ from the algorithm, to obtain\\
	\begin{equation}\label{eeeeppa}
	\begin{split}
	\varphi_k(\mu_{k+1})& \leq \varphi_k(\mu_k)+\theta_k\scal{A\minx_k-b}{Ax_{k+1}-b}+\frac{\theta_k^2}{2\rhoinf}\norrm{Ax_{k+1}-b}^2\\
	& \leq \varphi_k(\mu_k)+\frac{\theta_k}{2}\norrm{A\minx_k-b}^2+\frac{\theta_k}{2}\left(\frac{\theta_k}{\rhoinf}+1\right)\norrm{Ax_{k+1}-b}^2,
	\end{split}
	\end{equation}
	where we estimated the scalar product by Cauchy-Schwartz and Young inequality.
	Moreover, by definition, 
	\begin{equation}\label{prev}
	\begin{split}
	\varphi_{k+1}(\mu_{k+1})
	& = -\inf_{x\in\HH_p}  \left\{f(x)+g^{\beta_{k+1}}(Tx)+h(x)+\scal{\mu_{k+1}}{Ax-b}+\frac{\rho_{k+1}}{2}\norrm{Ax-b}^2\right\}\\
	& = \sup_{x\in\HH_p}  \left\{-\Lk{x,\mu_{k+1}}+\left[g^{\beta_{k}}-g^{\beta_{k+1}}\right](Tx)+\frac{1}{2}\left(\rho_k-\rho_{k+1}\right)\norrm{Ax-b}^2\right\}.
	\end{split}
	\end{equation}
	Now recall assumptions \ref{ass:beta} and \ref{ass:rhodec}: for $\beta_k$ non-increasing, $\left[g^{\beta_{k}}-g^{\beta_{k+1}}\right](Tx)\leq 0$ for every $x\in\HH_p$ by \propref{prop:moreau}\ref{moreauclaim4} and, for $\rho_k$ non-decreasing, $\rho_{k}-\rho_{k+1}\leq 0$. Then we can estimate the right-hand-side of \eqref{prev} to obtain
\begin{equation*}
	\begin{split}
		\varphi_{k+1}(\mu_{k+1}) \leq \sup_{x\in\HH_p}  \  -\Lk{x,\mu_{k+1}}= \varphi_{k}(\mu_{k+1}).
	\end{split}
\end{equation*}
	Sum \eqref{eeeeppa} with the latter, to obtain
	\begin{equation*}
	\begin{split}
	\varphi_{k+1}\left(\mu_{k+1}\right)-\varphi_{k}\left(\mu_{k}\right)& \leq\frac{\theta_k}{2}\norrm{A\minx_k-b}^2+\frac{\theta_k}{2}\left(\frac{\theta_k}{\rhoinf}+1\right)\norrm{Ax_{k+1}-b}^2.
	\end{split}
	\end{equation*}
	By Assumption \ref{ass:thetagamma}, $\theta_k=\gamma_k/c$ where $\gamma_k\leq 1$. Moreover, by assumption \ref{ass:gkgkp1}, $\gamma_{k}\leq \gamcon\gamma_{k+1}$. Then,
	\begin{equation}\label{final}
		\begin{split}
			\varphi_{k+1}\left(\mu_{k+1}\right)-\varphi_{k}\left(\mu_{k}\right)& \leq\frac{\gamma_k}{2c}\norrm{A\minx_k-b}^2+\frac{\gamcon}{2c}\left(\frac{1}{\rhoinf c}+1\right)\gamma_{k+1}\norrm{Ax_{k+1}-b}^2.
		\end{split}
	\end{equation}
	Notice that the right-hand-side is in $\ell^1_+$, because both $\seq{\gamma_{k}\norm{Ax_{k}-b}{}^2}$ and \linebreak $\seq{\gamma_{k}\norm{A\minx_{k}-b}{}^2}$ are in $\ell^1_+$ by \lemref{lem:feassum}. Additionally, $\seq{\varphi_k(\mu_k)}$ is bounded from below. Indeed, by virtue of \ref{ass:existence} and \remref{rem:assfun}\ref{rem:assfunproper}, we have
	\begin{equation*}
	\begin{split}
	\varphi_k(\mu_k) & \geq -\Lk{\xs,\mu_k}\\
	& \geq -\left[f(\xs)+g(T\xs)+h(\xs)\right] > -\infty .
	\end{split}
	\end{equation*}
	Then we can use \lemref{Alber2}\ref{alber2rk} on inequality \eqref{final} to conclude that $\seq{\varphi_{k}\para{\mu_{k}}}$ is convergent and, in particular, bounded. Now recall $\Phi_k$, $\bar{\Phi}$ and $\bar{\varphi}$ from~\eqref{notation}. Notice that
	\begin{equation*}
	\begin{split}
	\varphi_k(\mu) 
	& = \sup_{x\in\HH_p} \left\{\scal{\mu}{b-Ax}-\Phi_k(x)\right\}\\
	& = \sup_{x\in\HH_p} \left\{\scal{-A^*\mu}{x}-\Phi_k(x)\right\}+\scal{b}{\mu}\\
	& = \Phi_k^{*}\left(-A^*\mu\right)+\scal{b}{\mu} .
	\end{split}
	\end{equation*}
	It then follows that
	\begin{equation}\label{unif_coerc}
	\begin{split}
	& g^{\beta_k} \leq g \quad \implies \quad \Phi_k\leq\bar{\Phi} \quad \iff \quad  \bar{\Phi}^{*} \leq \Phi_k^{*} \quad \implies \quad \bar{\varphi} \leq \varphi_k,
	\end{split}
	\end{equation}
where we used \propref{prop:moreau}\ref{moreauclaim4} and the fact in \eqref{fact}.
	We are now in position to invoke \lemref{lem:coercive} which shows that $\bar{\varphi}$ is coercive on $\ran(A)$, and thus, by \eqref{unif_coerc}, $\seq{\varphi_k}$ is coercive uniformly in $k$ on $\ran(A)$. In turn, since $\ran(A)$ is closed and $\seq{\mu_k} \subset \ran(A)=\ker(A^*)^\bot$, we have from \eqref{unif_coerc} and the proof of \lemref{lem:coercive} that
	\[
	\exists (a > 0, \alpha > 0, \beta \in \R), (\forall k \in \N), \quad \varphi_k(\mu_k) \geq \bar{\varphi}(\mu_k) \geq a\norm{A^*\mu_k}{} + \beta \geq a\alpha\norm{\mu_k}{} + \beta ,
	\]
	which shows that $\seq{\mu_k}$ is indeed bounded by boundedness of $\seq{\varphi_{k}\para{\mu_k}}$.
\end{proof}
\subsection{Optimality}\label{sec:opt}

In this section we prove Theorem~\ref{convergence} by establishing convergence of the Lagrangian values to the optimum (i.e., the value at the saddle-point).


We start by showing some boundedness claims that will be important in our proof.

\begin{lemma}\label{lem:bndPhiMcstopt}
Under assumptions \ref{ass:A1}-\ref{ass:coercive} and \ref{ass:P1}-\ref{ass:thetagamma}, the objective $\Phi$ is bounded on $\C$, and thus
\begin{equation}\label{Mt}
\Mcstopt \eqdef \sup_{x\in\C} |\Phi(x)|+\sup_{k\in\N}\norrm{\mu_k} \left(\norrm{A} \ R+\norrm{b}\right) < \pinfty ,
\end{equation}
where we recall the radius $R$ from assumption \ref{ass:compact}.
\end{lemma}

\begin{proof}
By assumption~\ref{ass:interior}, $g$ is subdifferentiable at $Tx$ for any $x \in \C$. Thus convexity of $g$ implies that for any $x \in \C$
\nnewq{\label{eq:bndgC}
g(Tx) &\leq g(T\xs) + \ip{\sbrac{\p g\para{Tx}}^0,Tx-T\xs}{} \leq g(T\xs) + \norm{\sbrac{\p g\para{Tx}}^0}{}\norrm{T}\diam \\
g(Tx) &\geq g(T\xs) + \ip{\sbrac{\p g\para{T\xs}}^0,Tx-T\xs}{} \geq g(T\xs) - \norm{\sbrac{\p g\para{T\xs}}^0}{}\norrm{T}\diam .
}
From assumptions~\ref{ass:A1} and \ref{ass:f}, $f$ belongs to $\Gamma_0\para{\HH_p}$ and is differentiable on an open set $\C_0$ that contains $\C \subset \dom(f)$ (see \defref{def_smooth}). Thus the continuity set of $f$ contains $\C$, and it follows from \cite[Corollary~8.30(ii)]{BAUSCHCOMB} that $\C \subset \inte\para{\dom\para{f}}$. Consequently, arguing as in the proof of \lemref{lem:interior}, we deduce that 
\begin{equation}\label{D}
\sup\limits_{x\in \C}\norm{\nabla f\para{x}}{} < \pinfty.
\end{equation}
In turn, convexity entails that for any $x \in \C$
\nnewq{\label{eq:bndfC}
f(x) &\leq f(\xs) + \ip{\nabla f\para{x},x-\xs}{} \leq f(\xs) + \norm{\nabla f\para{x}}{}\diam, \\
f(x) &\geq f(\xs) + \ip{\nabla f\para{\xs},x-\xs}{} \geq f(\xs) - \norm{\nabla f\para{\xs}}{}\diam .
}
From assumption~\ref{ass:lip}, we also have for any $x \in \C$
\nnewq{\label{eq:bndhC}
h(\xs) - L_h\diam \leq h(x) \leq h(\xs) + L_h\diam .
}
Summing~\eqref{eq:bndgC}, \eqref{eq:bndfC} and \eqref{eq:bndhC}, using~\eqref{D} and assumption~\ref{ass:interior}, we get
\[
|\Phi(x)| \leq |\Phi(\xs)| + \para{L_h+\norrm{T}\sup_{x \in \C}\norm{\sbrac{\p g\para{Tx}}^0}{}+\sup_{x \in \C}\norm{\nabla f\para{x}}{}} .
\]
From Lemma \ref{bounded}, we know that the sequence of dual variables $\seq{\mu_k}$ is bounded which concludes the proof.
\end{proof}

Define $C_k\eqdef \frac{L_k}{2}\diam^2 + \diam\para{D + M\norrm{T} + L_h + \norrm{A}\ \norrm{\mus}}$, where $L_k$ is given in \eqref{Lk} and the constants $D$, $M$, and $L_h$ are as in \lemref{lemma:LkBound}. We then have the following lemma, in which we state the main energy estimation.
\begin{lemma}\label{estimate}
Suppose that assumptions \ref{ass:A1}-\ref{ass:coercive} and \ref{ass:P1}-\ref{ass:thetagamma} hold, with $\gamconinf \geq 1$. Consider the sequence of primal-dual iterates $\seq{(x_k,\mu_k)}$ generated by \algref{alg:CGAL} and $\left(\xs,\mus\right)$ a saddle-point point of the Lagrangian as in \eqref{saddle_point}.
Let
\begin{equation}\label{r}
r_k \eqdef (1-\gamma_k)\Lk{x_k,\mu_k}+\frac{c}{2}\norrm{\mu_{k}-\mus}^2+\frac{\beta_k}{2}\gcon^2+\gamma_k \Mcstopt.
	\end{equation}
	Then, we have the following energy estimate
	\begin{equation}\label{mainest}
	\begin{aligned}
	&r_{k+1}-r_k +\gamma_k \left[\LL{x_{k},\mus}-\LL{\xs,\mus}+\frac{\rho_k}{2}\norrm{Ax_{k}-b}^2\right]\leq\\ &\frac{1}{2}\left[\rho_{k+1}-\rho_{k}-\gamma_{k+1}\rho_{k+1}+\frac{2}{c}\gamma_k -\frac{\gamma_k^2}{c}\right]\norrm{Ax_{k+1}-b}^2
	+\frac{\gamma_k\beta_{k}}{2}\gcon^2 +K_{\left(F,\zeta,\C\right)}\zeta\left(\gamma_k\right)+C_k\gamma_k^2.
	\end{aligned}
	\end{equation}
\end{lemma}


\begin{proof}
	Notice that the dual update $\mu_{k+1} = \mu_k + \theta_k \para{Ax_{k+1}-b}$ can be re-written as
	\begin{equation*}
	\ens{\mu_{k+1}} = \Argmin_{\mu\in \HH_d} \ \left\{ -\Lk{x_{k+1},\mu} + \frac{1}{2\theta_k}\norrm{\mu-\mu_k}^2\right\}.
	\end{equation*}
	Then, from firm nonexpansiveness of the proximal mapping (see \eqref{non-exp}),
	\begin{equation}\label{prox_mu}
	\begin{split}
	0&\geq\theta_k\left[\Lk{x_{k+1},\mus}-\Lk{x_{k+1},\mu_{k+1}}\right] + \frac{1}{2}\big[\norrm{\mu_{k+1}-\mus}^2-\norrm{\mu_k-\mus}^2\\
	 & \qquad\qquad +\norrm{\mu_{k+1}-\mu_k}^2\big]\\
	&=\theta_k\left[\Lk{x_{k+1},\mus}-\Lk{x_{k+1},\mu_{k+1}}\right] + \frac{1}{2}\left[\norrm{\mu_{k+1}-\mus}^2-\norrm{\mu_k-\mus}^2\right] \\
	&\qquad\qquad +\frac{\theta_k^2}{2}\norrm{Ax_{k+1}-b}^2.
	\end{split}
	\end{equation}
	Notice that 
	\begin{equation*}
	\begin{split}
	\Lk{x_{k+1},\mu_{k}}- \Lk{x_k,\mu_k}=
	\left[\Ek{x_{k+1},\mu_{k}}+h(x_{k+1})\right]- \left[\Ek{x_k,\mu_k}+h(x_k)\right]
	\end{split}
	\end{equation*}
	and that, by the definition of $x_{k+1}$ in the algorithm and by convexity of function $h$,
	\begin{eqnarray*}
		h(x_{k+1})-h(x_k)&=&h(\left(1-\gamma_k\right)x_k+\gamma_k s_k)-h(x_k)\\
		&\leq&  \gamma_k \left(h(s_k)-h(x_k)\right).
	\end{eqnarray*}
Then,
\begin{equation}\label{lab}
	\begin{split}
		\Lk{x_{k+1},\mu_{k}}- \Lk{x_k,\mu_k}\leq
		\Ek{x_{k+1},\mu_{k}}-\Ek{x_k,\mu_k}+\gamma_k \left(h(s_k)-h(x_k)\right).
	\end{split}
\end{equation}
	Now apply \lemref{lemma:lowerbound} at the points $\xs$, $x_k$, and $\mu_k$ to affirm that
\begin{equation*}
\begin{split}
\Ek{\xs,\mu_k} & \geq \Ek{x_k,\mu_k} + \scal{\nabla_x \Ek{x_k,\mu_k}}{\xs-x_k}+\frac{\rho_k}{2}\norrm{A(\xs-x_k)}^2.
\end{split}
\end{equation*}
From the latter, by the alternative definition of $s_k$ in the algorithm (see \eqref{sk}), we obtain
\begin{equation}\label{strong_conv}
\begin{split}
\Ek{\xs,\mu_k}
& \geq \Ek{x_k,\mu_k} -h(\xs)+ h(s_k)+ \scal{\nabla_x \Ek{x_k,\mu_k}}{s_k-x_k}+\frac{\rho_k}{2}\norrm{Ax_k-b}^2.
\end{split}
\end{equation}
From Lemma \ref{DescLemma}, we have also that
\begin{equation*}
	\begin{split}
		\Ek{x_{k+1},\mu_k} & \leq \Ek{x_k,\mu_k} + \scal{\nabla_x \Ek{x_k,\mu_k}}{x_{k+1}-x_k} +  K_{\left(F,\zeta,\C\right)}\zeta\left(\gamma_k\right) + \frac{L_k}{2}\norrm{x_{k+1}-x_k}^2. 
	\end{split}
\end{equation*}
Recall that, from the algorithm, $x_{k+1}=x_k+\gamma_k\left(s_k-x_k\right)$. Then,
\begin{equation*}
	\begin{split}
		\Ek{x_{k+1},\mu_{k}}&\leq \Ek{x_k,\mu_k} + \gamma_k \scal{\nabla_x \Ek{x_k,\mu_k}}{s_k-x_k} +K_{\left(F,\zeta,\C\right)}\zeta\left(\gamma_k\right)+\frac{L_k\gamma_k^2}{2}\norrm{s_{k}-x_k}^2\\
		&\leq \Ek{x_k,\mu_k}+\gamma_k \left[\Ek{\xs,\mu_k}+h(\xs)-\Ek{x_k,\mu_k}-h(s_k)-\frac{\rho_k}{2}\norrm{Ax_k-b}^2\right] \\
		& \ \ \ +K_{\left(F,\zeta,\C\right)}\zeta\left(\gamma_k\right)+\frac{L_k}{2}\diam^2\gamma_k^2, \label{e}
	\end{split}
\end{equation*}
where in the last inequality we used \eqref{strong_conv}.
	Using the latter in \eqref{lab}, we obtain
	\begin{equation}\label{labb}
		\begin{split} \Lk{x_{k+1},\mu_{k}}-\Lk{x_k,\mu_k}	\leq & \gamma_k \left[\Lk{\xs,\mu_k}-\Lk{x_k,\mu_k}-\frac{\rho_k}{2}\norrm{Ax_{k}-b}^2\right]\\ &+K_{\left(F,\zeta,\C\right)}\zeta\left(\gamma_k\right)+\frac{L_k}{2}\diam^2\gamma_k^2.
		\end{split}
	\end{equation}
	Notice also that, from the definitions of $\Lk{x_{k+1},\cdot}$ and $\mu_{k+1}$ as $\mu_{k+1} = \mu_k + \theta_k\para{Ax_{k+1}-b}$,
	\begin{equation*}
	\Lk{x_{k+1},\mu_{k+1}}-\Lk{x_{k+1},\mu_{k}}=\scal{\mu_{k+1}-\mu_k}{Ax_{k+1}-b}=\theta_k\norrm{Ax_{k+1}-b}^2.
	\end{equation*}
	So, from the latter and \eqref{labb},
	\newqml{\label{estt}
		\Lk{x_{k+1},\mu_{k+1}} -\Lk{x_k,\mu_k}
		\leq\theta_k\norrm{Ax_{k+1}-b}^2+\gamma_k \left[\Lk{\xs,\mu_k}-\Lk{x_k,\mu_k}\right] \\
		-\frac{\rho_k\gamma_k}{2}\norrm{Ax_{k}-b}^2 + K_{\left(F,\zeta,\C\right)}\zeta\left(\gamma_k\right)+\frac{L_k}{2}\diam^2\gamma_k^2.
	}
	Now recall that, by assumption \ref{ass:thetagamma}, $\theta_k =  \gamma_k/c$. Multiply \eqref{prox_mu} by $c$ and sum with the latter, to obtain
	\begin{eqnarray*}
		& (1-c\theta_k)\Lk{x_{k+1},\mu_{k+1}} -(1-c\theta_k)\Lk{x_k,\mu_k} + \frac{c}{2}\left[\norrm{\mu_{k+1}-\mus}^2-\norrm{\mu_k-\mus}^2\right] \\
		& \leq \ \ \left(\theta_k -\frac{c\theta_k^2}{2}\right)\norrm{Ax_{k+1}-b}^2+\gamma_k \left[\Lk{\xs,\mu_k}-\Lk{x_k,\mu_k}\right]-c\theta_k \left[\Lk{x_{k+1},\mus}-\Lk{x_k,\mu_k}\right] \\
		& \ \ \ -\frac{\rho_k\gamma_k}{2}\norrm{Ax_{k}-b}^2+K_{\left(F,\zeta,\C\right)}\zeta\left(\gamma_k\right)+\frac{L_k}{2}\diam^2\gamma_k^2.
	\end{eqnarray*}
	The previous inequality can be re-written, by trivial manipulations, as
	\begin{equation}\label{est}
	\begin{split}
	& (1-c\theta_{k+1})\Lkp{x_{k+1},\mu_{k+1}} -(1-c\theta_k)\Lk{x_k,\mu_k} + \frac{c}{2}\left[\norrm{\mu_{k+1}-\mus}^2-\norrm{\mu_k-\mus}^2\right]\\
	\leq & \ \ (1-c\theta_{k+1})\Lkp{x_{k+1},\mu_{k+1}}-(1-c\theta_{k})\Lk{x_{k+1},\mu_{k+1}}+\left(\theta_k -\frac{c\theta_k^2}{2}\right)\norrm{Ax_{k+1}-b}^2\\
	& \ +\gamma_k \left[\Lk{\xs,\mu_k}-\Lk{x_k,\mu_k}\right]-c\theta_k \left[\Lk{x_{k+1},\mus}-\Lk{x_k,\mu_k}\right]-\frac{\rho_k\gamma_k}{2}\norrm{Ax_{k}-b}^2 \\
	& \ +K_{\left(F,\zeta,\C\right)}\zeta\left(\gamma_k\right)+\frac{L_k}{2}\diam^2\gamma_k^2\\
	=&\ \ c\left(\theta_k-\theta_{k+1}\right)\left[f+h+\scal{\mu_{k+1}}{A\cdot-b}\right](x_{k+1})+\left[\left(1-c\theta_{k+1}\right)g^{\beta_{k+1}}-\left(1-c\theta_k\right)g^{\beta_k}\right]\left(Tx_{k+1}\right)\\
	& \ + \frac{1}{2}\left[\left(1-c\theta_{k+1}\right)\rho_{k+1}-\left(1-c\theta_{k}\right)\rho_{k}+2\theta_k -c\theta_k^2\right]\norrm{Ax_{k+1}-b}^2\\
	& \ +\gamma_k \left[\Lk{\xs,\mu_k}-\Lk{x_k,\mu_k}\right]-c\theta_k \left[\Lk{x_{k+1},\mus}-\Lk{x_k,\mu_k}\right]-\frac{\rho_k\gamma_k}{2}\norrm{Ax_{k}-b}^2 \\
	& \ +K_{\left(F,\zeta,\C\right)}\zeta\left(\gamma_k\right)+\frac{L_k}{2}\diam^2\gamma_k^2.
	\end{split}
\end{equation}
By \ref{ass:gkgkp1} and \ref{ass:thetagamma}, and the assumption that $\gamconinf \geq 1$, we have $\theta_{k+1} \leq \gamconinf^{-1}\theta_k \leq \theta_k$. In view of \ref{ass:beta}, we also have $\beta_{k+1}\leq\beta_{k}$ by \ref{ass:beta}. In particular, $g^{\beta_{k}}\leq g^{\beta_{k+1}}\leq g$. Now, by Proposition \ref{prop:moreau}\ref{moreauclaim3} and the definition of the constant $\gcon$ in \eqref{MMM}, we are able to estimate the quantity
	\begin{equation*}
	\begin{split}	
	&\left[\left(1-c\theta_{k+1}\right)g^{\beta_{k+1}}-\left(1-c\theta_k\right)g^{\beta_k}\right]\left(Tx_{k+1}\right)\\
	&=\left[g^{\beta_{k+1}}-g^{\beta_k}\right]\left(Tx_{k+1}\right)+c\left[\theta_kg^{\beta_k}-\theta_{k+1}g^{\beta_{k+1}}\right]\left(Tx_{k+1}\right)\\
	&\leq \frac{1}{2}\left(\beta_k-\beta_{k+1}\right)\norrm{\left[\partial g(Tx_{k+1})\right]^0}^2+c\left[\theta_kg^{\beta_k}-\theta_{k+1}g^{\beta_{k}}\right]\left(Tx_{k+1}\right)\\
	&\leq \frac{1}{2}\left(\beta_k-\beta_{k+1}\right)\gcon^2+c\left(\theta_k-\theta_{k+1}\right)g(Tx_{k+1}).
	\end{split}
	\end{equation*}
	Then,
	\begin{equation}\label{est2}
	\begin{split}
	& \ \ c\left(\theta_k-\theta_{k+1}\right)\left[f+h+\scal{\mu_{k+1}}{A\cdot-b}\right](x_{k+1})+\left[\left(1-c\theta_{k+1}\right)g^{\beta_{k+1}}-\left(1-c\theta_k\right)g^{\beta_k}\right]\left(Tx_{k+1}\right)\\
	\leq &\ \ c\left(\theta_k-\theta_{k+1}\right)\LL{x_{k+1},\mu_{k+1}}+\frac{1}{2}\left(\beta_k-\beta_{k+1}\right)\gcon^2.
	\end{split}
	\end{equation}
	
	Recall that, by assumption \ref{ass:compact}, $\C$ is convex and bounded and that, by the update $x_{k+1}=x_{k}+\gamma_{k}\left(s_k-x_k\right)$ with $s_k\in \C$ and $\gamma_{k}\in]0,1]$ by \ref{ass:P1}, $x_k$ always belongs to $\C$. From the assumptions, the functions $f,h$ and $g\circ T$ are bounded on $\C$ and, from the algorithm and convexity, $\seq{x_{k}}\subset\C$. By Lemma \ref{bounded}, also the sequence $\seq{\mu_k}$ is bounded. Then, recalling $\Mcstopt$ from \lemref{lem:bndPhiMcstopt}, we can use the Cauchy-Schwartz and the triangular inequality to affirm that
	\begin{equation}\label{multt}
	\begin{split}
	\LL{x_{k},\mu_{k}} = \Phi(x_k)+\scal{\mu_k}{Ax_k-b}
	\leq\Mcstopt.
	\end{split}
	\end{equation}
	Recall the definition of $r_k$ in \eqref{r}. Coming back to \eqref{est} and using both \eqref{est2} and \eqref{multt}, we obtain
	\begin{equation}\label{est3}
	\begin{split}
	r_{k+1}-r_k
	\ & \leq \ \frac{1}{2}\left[\left(1-\gamma_{k+1}\right)\rho_{k+1}-\left(1-\gamma_{k}\right)\rho_{k}+\frac{2}{c}\gamma_k -\frac{\gamma_k^2}{c}\right]\norrm{Ax_{k+1}-b}^2\\
	& \ \  +\gamma_k \left[\Lk{\xs,\mu_k}-\Lk{x_{k+1},\mus}\right]-\frac{\rho_k\gamma_k}{2}\norrm{Ax_{k}-b}^2 +K_{\left(F,\zeta,\C\right)}\zeta\left(\gamma_k\right)+\frac{L_k}{2}\diam^2\gamma_k^2.
	\end{split}
	\end{equation}
	Recall that, by feasibility of $\xs$, $\LL{\xs,\mu_k}=\LL{\xs,\mus}$. Now compute
	\begin{equation*}
	\begin{split}
	\Lk{\xs,\mu_k}-\Lk{x_{k+1},\mus} & = \LL{\xs,\mu_k}-\LL{x_{k+1},\mus}+\left[g^{\beta_k}-g\right](T\xs)+ \left[g-g^{\beta_{k}}\right](Tx_{k+1})\\
	&\quad\quad-\frac{\rho_k}{2}\norrm{Ax_{k+1}-b}^2\\
	&\leq \LL{\xs,\mus}-\LL{x_{k+1},\mus}+\frac{\beta_{k}}{2}M^2-\frac{\rho_k}{2}\norrm{Ax_{k+1}-b}^2,
	\end{split}
	\end{equation*}
	where in the inequality we used the facts that $g^{\beta_k}\leq g$ and that, by Proposition \ref{prop:moreau}\ref{moreauclaim4} and \eqref{MMM}, 
	$$\left[g-g^{\beta_{k}}\right](Tx_{k+1}) \leq \frac{\beta_{k}}{2}\norrm{\left[\partial g(Tx_{k+1})\right]^0}^2 \leq\frac{\beta_{k}}{2}M^2.$$
	Then, using the latter in \eqref{est3}, we obtain
	\begin{multline*}
		r_{k+1}-r_k
		\leq \frac{1}{2}\left[\rho_{k+1}-\rho_{k}-\gamma_{k+1}\rho_{k+1}+\frac{2}{c}\gamma_k -\frac{\gamma_k^2}{c}\right]\norrm{Ax_{k+1}-b}^2 +\gamma_k \left[\LL{\xs,\mus}-\LL{x_{k+1},\mus}\right]\\
		\quad+\frac{\gamma_k\beta_{k}}{2}M^2-\frac{\rho_k\gamma_k}{2}\norrm{Ax_{k}-b}^2 +K_{\left(F,\zeta,\C\right)}\zeta\left(\gamma_k\right)+\frac{L_k}{2}\diam^2\gamma_k^2.
	\end{multline*}
	We replace the term $\left[\LL{\xs,\mus}-\LL{x_{k+1},\mus}\right]$ with $\left[\LL{\xs,\mus}-\LL{x_{k},\mus}\right]+\left[\LL{x_k,\mus}-\LL{x_{k+1},\mus}\right]$ and estimate using \lemref{lemma:LkBound} to get the following,
	\begin{multline*}
		r_{k+1}-r_k
		\leq \frac{1}{2}\left[\rho_{k+1}-\rho_{k}-\gamma_{k+1}\rho_{k+1}+\frac{2}{c}\gamma_k -\frac{\gamma_k^2}{c}\right]\norrm{Ax_{k+1}-b}^2 +\gamma_k \left[\LL{\xs,\mus}-\LL{x_{k},\mus}\right]\\
		+\frac{\gamma_k\beta_{k}}{2}M^2-\frac{\rho_k\gamma_k}{2}\norrm{Ax_{k}-b}^2 +K_{\left(F,\zeta,\C\right)}\zeta\left(\gamma_k\right)+C_k\gamma_k^2.
	\end{multline*}
	We conclude by trivial manipulations.
\end{proof}

We are now ready to prove Theorem~\ref{convergence}.
\begin{proof}
Our starting point is the main energy estimate~\eqref{mainest}. Let us focus on its right-hand-side. Under assumption \ref{ass:cond}, 
\newq{
\frac{1}{2}\left[\rho_{k+1}-\rho_{k}-\gamma_{k+1}\rho_{k+1}+\frac{2}{c}\gamma_k -\frac{\gamma_k^2}{c}\right]\norrm{Ax_{k+1}-b}^2\leq \gamma_{k+1}\norrm{Ax_{k+1}-b}^2,
}
	where the right hand side is in $\ell^1_+$ by \lemref{lem:feassum}.
	Now remember that $C_k=\frac{L_k}{2}\diam^2 + \diam\para{D + M\norrm{T} + L_h + \norrm{A}\ \norrm{\mus}}$, where $L_k=\norrm{T}^2/\beta_k+\norrm{A}^2\rho_k$. Then we have
	\newqml{
	\gamma_k\beta_{k}M^2/2 +K_{\left(F,\zeta,\C\right)}\zeta\left(\gamma_k\right)+C_k\gamma_k^2
	= \gamma_k\beta_{k}M^2 /2 +K_{\left(F,\zeta,\C\right)}\zeta\left(\gamma_k\right) +\norrm{T}^2\gamma_k^2\diam /\left(2\beta_k\right)+\norrm{A}^2\rho_k\gamma_k^2\diam /2\\
	+ \diam\para{D + M\norrm{T} + L_h + \norrm{A}\ \norrm{\mus}}\gamma_k^2 \in \ell_+^1.
	}
	Indeed, under assumption \ref{ass:P1}, the sequences $\seq{\gamma_k\beta_k}, \seq{\zeta\left(\gamma_k\right)}$, and $\seq{\gamma_k^2/\beta_k}$ belong to $\ell_+^1$. Moreover, we have by assumptions \ref{ass:beta} and \ref{ass:rhodec} that $\rhoinf\gamma_k^2\leq \rho_k\gamma_k^2 \leq \beta_{0}\rhosup\gamma_k^2/\beta_k$, whence we get that $\seq{\rho_k\gamma_k^2} \in \ell_+^1$ and $\seq{\gamma_k^2}\in\ell_+^1$ after invoking assumption~\ref{ass:P1}. Thus all terms on the right hand side are summable. Let 
	\newq{
	w_k &\eqdef \left[\LL{x_{k},\mus}-\LL{\xs,\mus}\right]+\frac{\rho_k}{2}\norrm{Ax_{k}-b}^2 \\
	z_k &\eqdef \gamma_{k+1}\norrm{Ax_{k+1}-b}^2 + \gamma_k\beta_{k}M^2/2 +K_{\left(F,\zeta,\C\right)}\zeta\left(\gamma_k\right)+C_k\gamma_k^2 .
	}
	So far, we have shown that
	\nnewq{\label{eq:mainenerseq}
	r_{k+1} \leq r_k - \gamma_k w_k + z_k ,
	}
	where $r_k$ is bounded from below, and $\seq{z_k} \in \ell_+^1$. The rest of the proof consists of invoking properly \lemref{Alber2}.
	
	\begin{enumerate}[label=(\roman*)]
	\item In order to use \lemref{Alber2}\ref{alber2wk}, we need to show that for some positive constant $\alpha$,
	\newq{
		w_k-w_{k+1}\leq \alpha \gamma_k .
	}	
	Notice that the term $ \LL{x_{k},\mus}-\LL{\xs,\mus}$ is proportional to $\gamma_k$ by \lemref{lemma:LkBound}.
	For the second term of $w_k$, we have by \lemref{lem:feasgammabound} that $\frac{\rho_k}{2}\norm{Ax_k-b}{}^2 - \frac{\rho_{k+1}}{2}\norm{Ax_{k+1}-b}{}^2$ is proportional to $\gamma_k$. The desired claim then follows from \lemref{Alber2}\ref{alber2wk}.
	
	\item By \cite[Lemma~2.37]{BAUSCHCOMB}, we can assert that $\seq{x_k}$ possesses a weakly convergent subsequence, say $\subseq{x_{k_j}}$, with cluster point $\bar{x}\in \C$. Since $\norm{A\cdot-b}{} \in \Gamma_0(\HH_p)$ and in view of \cite[Theorem~9.1]{BAUSCHCOMB}, we have 
\newq{
\norm{A\bar{x}-b}{} \leq \liminf_j\norm{A x_{k_j}-b}{} = \lim_k\norm{A x_k-b}{} = 0,
}
where we used lower semicontinuity of the norm and \thmref{convergence}. Thus $A\bar{x}=0$, meaning that $\bar{x}$ is a feasible point of \eqref{PProb}. In turn, $\mc{L}\para{\bar{x}, \mus} = \Phi(\bar{x})$. The function $\mc{L}\para{\cdot,\mus}$ is lower semicontinuous by \ref{ass:A1} and \ref{ass:existence}. Thus, using \cite[Theorem~9.1]{BAUSCHCOMB} and by virtue of claim (i), we have
\newq{
\Phi(\bar{x}) = \mc{L}\para{\bar{x},\mus}\leq\liminf_j\mc{L}\para{x_{k_j},\mus}=\lim_k \mc{L}\para{x_{k},\mus}=\mc{L}\para{\xs,\mus}\leq \mc{L}\para{x,\mus}
}
for all $x\in\HH_p$, and in particular for all $x \in A^{-1}(b)$. Thus, for every $x \in A^{-1}(b)$, we deduce that
\newq{
\Phi(\bar{x}) \leq \mc{L}\para{x,\mus} = \Phi(x) ,
}
meaning that $\bar{x}$ is a solution for problem \eqref{PProb}.\\

Meanwhile, as the sequence $\seq{\mu_k}$ is bounded by \lemref{bounded}, we can again invoke \cite[Lemma~2.37]{BAUSCHCOMB} to extract a weakly convergent subsequence $\subseq{\mu_{k_j}}$ with cluster point $\bar{\mu}$. By Fermat's rule (\cite[Theorem~16.2]{BAUSCHCOMB}), the weak  sequential cluster point $\bar{\mu}$ is a solution to \eqref{DProb} if and only if 
	\[
	0\in\p \para{\Phi^*\circ \para{-A^*}}\para{\bar{\mu}} + b . 
	\]
	Since the proximal operator is the resolvent of the subdifferential, it follows that \eqref{eq:gradphik} is equivalent to
	\nnewq{\label{eq:monincprox}
	\nabla\varphi_{k_j}\para{\mu_{k_j}}-b &\in \p\para{\Phi_{k_j}^*\circ\para{-A^*}}\para{\mu_{k_j} - \rho_{k_j}\nabla\varphi_{k_j}\para{\mu_{k_j}}} .
	}
	By \lemref{lem:feassum} it follows that $A\tilde x_k$ converges strongly to $b$ and, combined with \eqref{eq:gradphikxk}, thus $\nabla\varphi_{k_j}\para{\mu_{k_j}}$ converges strongly to $0$. On the other hand, $\mu_{k_j} - \rho_{k_j}\nabla\varphi_{k_j}\para{\mu_{k_j}}$ converges weakly to $\bar{\mu}$. We now argue that we can pass to the limit in \eqref{eq:monincprox} by showing sequential closedness.
	
	When $g \equiv 0$, we have, for all $j\in\N$, $\Phi_{k_j} \equiv f+h$ and the rest of the argument relies on sequential closedness of the graph of the subdifferential of $\Phi^*\circ(-A^*) \in \Gamma_0(\mc{H}_d)$ in the weak-strong topology. For the general case, our argument will rely on the fundamental concept of Mosco convergence of functions, which is epigraphical convergence for both the weak and strong topology (see~\cite{brezis72} and \cite[Definition~3.7]{attouch}). 
	
	By \propref{prop:moreau}\ref{moreauclaim4} and assumptions \ref{ass:A1}-\ref{ass:f}, $\subseq{\Phi_{k_j}}$ is an increasing sequence of functions in $\Gamma_0\para{\mc{H}_d}$. It follows from \cite[Theorem~3.20(i)]{attouch} that $\Phi_{k_j}$ Mosco-converges to $\sup_{j \in \N} \Phi_{k_j} = \sup_{j \in \N}f+g^{\beta_{k_j}} \circ T+h = f+g \circ T+h = \Phi$ since $\beta_{k_j} \to 0$ by \ref{ass:beta}. Bicontinuity of the Legendre-Fenchel conjugation for the Mosco convergence (see \cite[Theorem~3.18]{attouch}) entails that $\Phi_{k_j}^* \circ -(A^*)$ Mosco-converges to $(f+g \circ T+h)^*\circ(-A^*)=\Phi^*\circ(-A^*)$. This implies, via \cite[Theorem~3.66]{attouch}, that $\partial \Phi_{k_j}^* \circ (-A^*)$ graph-converges to $\partial \Phi^*\circ(-A^*)$, and \cite[Proposition~3.59]{attouch} shows that $\subseq{\partial \Phi_{k_j} \circ (-A^*)}$ is sequentially closed for graph-convergence in the weak-strong topology on $\mc{H}_d$, i.e., for any sequence $(v_{k_j},\eta_{k_j})$ in the graph of $\partial \Phi_{k_j}^*\circ(-A^*)$ such that $v_{k_j}$ converges weakly to $\bar{v}$ and $\eta_{k_j}$ converges strongly to $\bar{\eta}$, we have $\bar{\eta} \in \partial \Phi^*\circ(-A^*)(\bar{v})$. Taking $v_{k_j} = \nabla\varphi_{k_j}\para{\mu_{k_j}} - b$ and $\eta_{k_j} = \mu_{k_j} - \rho_{k_j}\nabla\varphi_{k_j}\para{\mu_{k_j}}$, we conclude that
	\newq{
	0 \in \p\para{\Phi^*\circ\para{-A^*}}\para{\bar{\mu}}+b ,
	}
	i.e., $\bar{\mu}$ is a solution of the dual problem~\eqref{DProb}.
	
	Recall $r_k$ from~\eqref{r} which verifies~\eqref{eq:mainenerseq}. From \lemref{Alber2}\ref{alber2rk}, $\seq{r_k}$ is convergent. By \ref{ass:P1} and \ref{ass:beta}, $\gamma_k$ and $\beta_k$ both converge to $0$. We also have that
	\newq{
	-\Lk{x_k,\mu_k} &= \para{\mc{L}(x_k,\mus) - \Lk{x_k,\mu_k}} - \mc{L}(x_k,\mus) \\
			&= g(T x_k) - g^{\beta_k}(T x_k) + \ip{\mus-\mu_k, Ax_k-b}{} - \frac{\rho_k}{2}\norm{Ax_k-b}{}^2 \\
			&\qquad - \mc{L}(x_k,\mus) .
	}
	We have from \thmref{thm:feas}(i) that $\frac{\rho_k}{2}\norm{Ax_k-b}{}^2 \to 0$. In turn, $\ip{\mus-\mu_k, Ax_k-b}{} \to 0$ since $\seq{\mu_k}$ is bounded (Lemma~\ref{bounded}). We also have $\mc{L}(x_k,\mus) \to \mc{L}(\xs,\mus)$ by claim (i) above. By \propref{prop:moreau}\ref{moreauclaim4} and \eqref{MMM}, we get that
	\newq{
	0 \leq \para{g(T x_k) - g^{\beta_k}(T x_k)} \leq \frac{\beta_{k}}{2} \gcon^2 .
	}
	Passing to the limit and in view of \ref{ass:beta}, we conclude that $g(T x_k) - g^{\beta_k}(T x_k) \to 0$. Altogether, this shows that $\Lk{x_k,\mu_k} \to \mc{L}(\xs,\mus)$. In turn, we conclude that the limit
	\[
	\lim_{k \to \infty}\norm{\mu_k-\mus}{}^2 = 2/c\para{\lim_{k \to \infty} r_k - \mc{L}(\xs,\mus)} 
	\]
	exists. Since $\mus$ was an arbitray optimal dual point, and we have shown above that each subsequence of $\seq{\mu_{k}}$ converges weakly to an optimal dual point, we are in position to invoke Opial's lemma \cite{opial1967weak} to conclude that the whole dual multiplier sequence weakly converges to a solution of the dual problem.

	\item Recalling that $\seq{\gamma_k}\not\in\ell^1_+$ (see assumption \ref{notell1}), the rates in \eqref{pointrateopt} follow by applying \lemref{Alber2}\ref{alber2wkrateinf}-\ref{alber2wkrate} to \eqref{eq:mainenerseq}. Notice that both terms in $w_k$ are positive and that $\rho_k\geq \rhoinf >0$ (see again assumption~\ref{ass:rhodec}). Therefore we have that, for the same subsequence $\subseq{x_{k_j}}$, \eqref{rate} holds.

	\item The ergodic rate \eqref{ergratefeas} follows by applying the Jensen's inequality to the convex function $\mc{L}\para{\cdot,\mus}$. 

	\end{enumerate}
	
\end{proof}

\section{Applications}\label{sec:app}

\subsection{Sum of several nonsmooth functions}\label{subsec:nnonsmooth}
In this section we explore the applications of \algref{alg:CGAL} to splitting in composite optimization problems, where we allow the presence of more than one nonsmooth function $g$ or $h$ in the objective:
\nnewq{\label{prod}
	\min\limits_{x\in \HH_p}\brac{f\para{x} + \sum\limits_{i=1}^n g_i\para{T_i x} + \sum\limits_{i=1}^n h_i\para{x}}.
}
First, we denote the product space by $\pspace \eqdef \mc{H}_p^n$ endowed with the scalar product $\bsip{\bs{x}}{\bs{y}}=\frac{1}{n}\sum_{i=1}^n\ip{\comp{x}{i},\comp{y}{i}}{}$, where $\bs{x}$ and $\bs{y}$ are vectors in $\pspace$ with $\bs{x} \eqdef  \pmat{\comp{x}{1}, \dots,\comp{x}{n}}^\top$. We define also $\V$ as the diagonal subspace of $\pspace$, i.e. $\V \eqdef  \small\{ \bs{x}\in\pspace: \comp{x}{1}=\ldots=\comp{x}{n} \small\}$, $\V^\perp$ the orthogonal subspace to $\V$, and $\Pi_{\V}, \PVo$ the orthogonal projections onto $\V$, $\V^\perp$ - respectively. We finally introduce the (diagonal) linear operator $\bs{T}: \pspace \to \pspace$ defined by 
\[
\left[\bs{T}\para{\bs{x}}\right]^{\left(i\right)}=T_i \comp{x}{i}
\]
and the functions
\newq{
F\para{\bs{x}} \eqdef  \frac{1}{n}\sum\limits_{i=1}^n f\para{x^{(i)}};\quad
G\para{\bs{T}\bs{x}} \eqdef  \sum\limits_{i=1}^n g_i\para{T_i x^{(i)}};\quad
H\para{\bs{x}} \eqdef  \sum\limits_{i=1}^n h_i\para{\comp{x}{i}}.
}
Then problem \eqref{prod} is obviously equivalent to
\nnewq{\label{prodequiv}
	\min\limits_{\bs{x}\in \pspace} \brac{F\para{\bs{x}} + G (\bs{T}\bs{x})+H\para{\bs{x}}: \ \ \PVo \bs{x} = 0},
}
which fits in the setting of our main problem \eqref{PProb}.
In order to make more clear the presentation, we separate the two cases of multiple $g$ and multiple $h$, that can be trivially combined. Moreover, we focus on the main case $h_i=\iota_{\C_i}$.

\subsection{Sum of several simple functions over a compact set}
Consider the following composite minimization problem,
\nnewq{\label{manyg}
\min\limits_{x\in \C}\brac{f\para{x} + \sum\limits_{i=1}^n g_i\para{T_i x}}.
}
We can reformulate the problem in the product space $\pspace$ using the above notation to get,
\newq{
\min\limits_{\bs{x}\in \C^n\cap \V} \brac{F\para{\bs{x}} + G\para{\bs{T}\bs{x}}}.
}
Applying \algref{alg:CGAL} to this problem gives a completely separable scheme; we first compute the direction,
\newq{
\bs{s}_k &\in \Argmin\limits_{\bs{s}\in \C^n\cap \V}\bsip{\nabla\para{F\para{\bs{x}_k} + G^{\beta_k}\para{\bs{T}\bs{x}_k}}}{\bs{s}},
}
which reduces to the following computation since $\bs{s}_k=\pmat{s_k\\\vdots\\ s_k}$ has identical components,
\newq{
s_k \in \Argmin\limits_{s\in \C}\ip{\sum\limits_{i=1}^n\para{\frac{1}{n}\nabla f\para{\comp{x_k}{i}} + \nabla g_i^{\beta_k}\para{T_i \comp{x_k}{i}}}, s}{}.
}
The term $\nabla g_i^{\beta_k}$ has a closed form given in \propref{prop:moreau} which can be used to get the following formula for the direction,
\newq{
s_k \in \Argmin\limits_{s\in \C}\ip{\sum\limits_{i=1}^n\para{\frac{1}{n}\nabla f\para{\comp{x_k}{i}} + \frac{1}{\beta_k}T_i^*\para{T_i \comp{x_k}{i}-\prox_{\beta g}\para{T_i\comp{x_k}{i}}}}, s}{}.
}

\subsection{Minimizing over intersection of compact sets}\label{app:manyh}
A classical problem found in machine learning is to minimize a Lipschitz-smooth function $f$ over the intersection of convex, compact sets $\C_i$ in some real Hilbert space $\mc{H}$,
\newq{
\min\limits_{x\in\bigcap\limits_{i=1}^n \C_i}f\para{x} = \min\limits_{x\in\mc{H}}\brac{f\para{x} + \sum\limits_{i=1}^n h_i\para{x}},
}
where $h_i\equiv \iota_{\C_i}$. Reformulating the problem in the product space $\pspace$ gives,
\newq{
\min\limits_{\sst{\bs{x}\in\pspace\\ \PVo\bs{x}=0}}\brac{ F\para{\bs{x}} + H\para{\bs{x}}}.
}
Then, we can apply \algref{alg:CGAL} and compute the step direction
\newq{
\bs{s}_k \in \Argmin\limits_{\bs{s}\in\C_1\times\ldots\times\C_n}\bsip{\bs{s}}{\nabla \sbrac{F\para{\bs{x}} + \ip{\bs{\mu}_k, \PVo \bs{x}_k} + \frac{\rho_k}{2}\norm{\PVo\bs{x}_k}{}^2}}
}
which gives a separable scheme for each component of $\bs{s}_k = \pmat{\comp{s_k}{1}\\\vdots\\\comp{s_k}{n}}$,
\nnewq{\label{eq:cgaldir}
\comp{s_k}{i} &\in \Argmin\limits_{s\in\C_i}\ip{s, \frac{1}{n}\nabla f\para{\comp{x_k}{i}} + \pcomp{\PVo\bs{\mu}_k}{i}+\rho_k\pcomp{\PVo \bs{x}_k}{i}}{}\\
&=\Argmin\limits_{s\in\C_i} \ip{s, \frac{1}{n}\nabla f\para{\comp{x_k}{i}} +\comp{\mu_k}{i}-\frac{1}{n}\sum\limits_{j=1}^n\comp{\mu_k}{j} +\rho_k\para{\comp{x_k}{i}-\frac{1}{n}\sum\limits_{j=1}^n\comp{x_k}{j}}}{}.
}

\begin{section}{Comparison}\label{sec:comparison}
\begin{subsection}{Conditional Gradient Framework}
In \cite{Yurt} the following problem was analyzed in the finite-dimensional setting,
\nnewq{\label{eq:yurtprob}
\min\limits_{x\in\C} \brac{f\para{x} + g\para{Tx}}
}
where $f\in C^{1,1}\para{\R^n}\cap\Gamma_0\para{\R^n}$, $T\in\R^{d\times n}$ is a linear operator, $g\circ T\in\Gamma_0\para{\R^n}$, and $\C$ is a compact, convex subset of $\R^n$. They develop an algorithm which avoids projecting onto the set $\C$, instead utilizing a linear minimization oracle $\lmo_{\C}\para{v} = \Argmin\limits_{x\in\C}\ip{x,v}{}$, and replaces the function $g\circ T$ with the smooth function $g^\beta_k\circ T$. They consider only functions $f$ which are Lipschitz-smooth and finite dimensional spaces, i.e. $\R^n$, compared to \cgal which weakens the assumptions on $f$ to be differentiable and $\para{F,\zeta}$-smooth (see \defref{def_smooth}) with an arbitrary real Hilbert space $\mc{H}_p$ (possibly infinite dimensional). Furthermore, the analysis in \cite{Yurt} is restricted to the parameter choices $\gamma_k = \frac{2}{k+1}$ and $\beta_k=\frac{\beta_0}{\sqrt{k+1}}$ exclusively, although they do include a section in which they consider two variants of an inexact linear minimization oracle: one with additive noise and one with multiplicative noise. In contrast, the results we present in \secref{sec:alg} show optimality and feasibility for a wider choice for both the sequence of stepsizes $\seq{\gamma_k}$ and the sequence of smoothing parameters $\seq{\beta_k}$, although we only consider exact linear perturbation oracles of the form $\Argmin\limits_{s\in\HH_p}\brac{h\para{s}+\ip{x,s}{}}$. Finally, for solving \eqref{eq:yurtprob} with an exact linear minimization oracle, our algorithm encompasses the algorithm in \cite{Yurt} by choosing $h\para{x} = \iota_\C\para{x}$, $A\equiv 0$, and restricting $f$ to be in $C^{1,1}\para{\mc{H}}$ with $\mc{H} = \R^n$.

In \cite[Section 5]{Yurt} there is a discussion on splitting and affine constraints using the conditonal gradient framework presented. In this setting, i.e. assuming exact oracles, the primary difference between \cgal and the conditional gradient framework is the approach each algorithm takes to handle affine constraints. In \cgal, the augmented Lagrangian formulation is used to account for the affine constraints, introducing a dual variable $\mu$ and both a linear and quadratic term for the constraint $Ax-b=0$. In contrast, in \cite{Yurt} the affine constraint is treated the same as the nonsmooth term $g \circ T$ and thus handled by quadratic penalization/smoothing alone. The consequence of smoothing for the affine constraint $Ax=b$ comes from calculating the gradient of the squared-distance to the constraint. This will involve solving a least squares problem at each iteration which can be computationally expensive. Our algorithm does not need to solve such a linear system. 

The difference in the approaches is highlighted when both methods are applied to problem presented in \secref{app:manyh} with $n=2$ since this problem necessitates an affine constraint $\PVo\bs{x}=0$ for splitting. According to \cite[Section~5]{Yurt}, we reformulate the problem to be
\newq{
\min\limits_{\sst{\comp{x}{1}\in\C_1\\\comp{x}{2}\in C_2}}\brac{\frac{1}{2}\para{f\para{\comp{x}{1}}+f\para{\comp{x}{2}}} + \iota_{\brac{\comp{x}{1}}}\para{\comp{x}{2}}}.
}
Note that the inclusion of the function $\iota_{\brac{\comp{x}{1}}}\para{\comp{x}{2}}$ in the objective is equivalent to the affine constraint $\PVo \bs{x}=0$ in the $n=2$ case. Apply the conditional gradient framework on the variable $\para{\comp{x}{1},\comp{x}{2}}$ to get
\newq{
\bs{s}_k &\in \Argmin\limits_{\sst{\comp{s}{1}\in \C_1\\ \comp{s}{2}\in \C_2}}\brac{\ip{\pmat{\comp{s}{1}\\ \comp{s}{2}},\pmat{\nabla_{\comp{x}{1}} \sbrac{\frac{1}{2}f\para{\comp{x_k}{1}} + \iota_{\comp{x_k}{2}}^{\beta_k}\para{\comp{x_k}{1}}}\\ \nabla_{\comp{x}{2}}\sbrac{\frac{1}{2}f\para{\comp{x_k}{2}}+\iota_{\comp{x_k}{1}}^{\beta_k}\para{\comp{x_k}{2}}}}}{}},
}
which leads to a separable scheme that can be computed component-wise,
\nnewq{\label{eq:yurtdir}
\comp{s_k}{1} &\in \Argmin\limits_{s\in\C_1} \ip{s, \frac{1}{2}\nabla f\para{\comp{x_k}{1}} + \frac{\comp{x_k}{1}-\comp{x_k}{2}}{\beta_k}}{}\\
\comp{s_k}{2} &\in \Argmin\limits_{s\in\C_2} \ip{s, \frac{1}{2}\nabla f\para{\comp{x_k}{2}} + \frac{\comp{x_k}{2}-\comp{x_k}{1}}{\beta_k}}{}.
}
Compare the direction obtained in \eqref{eq:yurtdir} to the one obtained in \eqref{eq:cgaldir}, the components of which we rewrite below for $n=2$,
\nnewq{\label{eq:cgaldir2}
\comp{s_k}{1} &\in \Argmin\limits_{s\in \C_1}\ip{s, \frac{1}{2} \nabla f\para{\comp{x_k}{1}} + \frac{1}{2}\para{\comp{\mu_k}{1}-\comp{\mu_k}{2}} + \frac{\rho_k}{2}\para{\comp{x_k}{1} - \comp{x_k}{2}}}{}\\
\comp{s_k}{2} &\in \Argmin\limits_{s\in \C_2}\ip{s, \frac{1}{2} \nabla f\para{\comp{x_k}{2}} + \frac{1}{2}\para{\comp{\mu_k}{2}-\comp{\mu_k}{1}} + \frac{\rho_k}{2}\para{\comp{x_k}{2} - \comp{x_k}{1}}}{}.
}
Due to affine constraint, the computation of the direction in \eqref{eq:yurtdir} necessitates smoothing and, as a consequence, the parameter $\beta_k$, which is necessarily going to $0$. In \cgal, the introduction of the dual variable $\bs{\mu}_k$ in place of smoothing the affine constraint avoids the parameter $\beta_k$. Instead, we have the parameter $\rho_k$ but $\rho_k$ can be picked to be constant without issue.

\end{subsection}

\begin{subsection}{FW-AL Algorithm}
In \cite{Gidel} the following problem was analyzed,
\newq{
\min\limits_{\sst{x\in \bigcap\limits_{i=1}^n \C_i\\Ax=0}} f\para{x}
}
using a combination of the Frank-Wolfe algorithm with the augmented Lagrangian to account for the constraint $Ax=0$. The function $f$ is assumed to be Lipschitz-smooth, in contrast to our approach. The perspective used in their paper is to modify the classic ADMM algorithm, replacing the marginal minimization with respect to the primal variable by a Frank-Wolfe step instead, although their analysis is not restricted only to Frank-Wolfe steps. Indeed, in all the scenarios where one can apply FW-AL using a Frank-Wolfe step our algorithm encompasses FW-AL as a special case, discussed in \secref{app:manyh}. The primary differences between \cgal and FW-AL are in the convergence results and the generality of \cgal. The results in \cite{Gidel} prove convergence of the objective in the case where the sets $\C_i$ are polytopes and convergence of the iterates in the case where the sets $\C_i$ are polytopes and $f$ is strongly convex, but they do not prove convergence of the objective, weak convergence of the dual variable, or asymptotic feasibility of the iterates in the general case where each $\C_i$ is a compact, convex set. Instead, they prove two theorems which imply subsequential convergence of the objective and subsequential asymptotic feasibility in the general case and subsequential convergence of the iterates to the optimum in the strongly convex case in \cite[Theorem 2]{Gidel} and \cite[Corollary 2]{Gidel} respectively. Unfortunately, each of these results is obtained separately and so the subsequences that produce each result are not guaranteed to coincide with one another.

Interestingly, the results they obtain are not unique to Frank-Wolfe style algorithms as their analysis is from the perspective of a modified ADMM algorithm; they only require that the algorithm used to replace the marginal minimization on the primal variable in ADMM produces sublinear decrease in the objective. Finally, they do not provide conditions for the dual multiplier sequence, $\mu_k$ in our notation, to be bounded as they discuss in their analysis of issues with similar proofs, e.g. in GDMM. This is a crucial issue as the constants in their bounds depend on the norm of these dual multipliers.
\end{subsection}

\end{section}

\section{Numerical Experiments}\label{sec:num}

In this section we present some numerical experiments comparing the performance of \algref{alg:CGAL} and a proximal algorithm applied to splitting in composite optimization problems. 

\subsection{Projection problem} 
First, we consider a simple projection problem,
\nnewq{\label{eq:projprob}
\min\limits_{x\in\R^2}\enscondlr{\frac{1}{2}\norm{x-y}{2}^2}{\norm{x}{1}\leq 1, Ax=0},
}
where $y\in\R^2$ is the vector to be projected and $A:\R^2\to\R^2$ is a rank-one matrix. To exclude trivial projections, we choose randomly $y \notin \Ball_1^1 \cap \ker(A)$, where $\Ball_1^1$ is the unit $\ell^1$ ball centered at the origin. Then Problem~\eqref{eq:projprob} is nothing but Problem~\eqref{PProb} with $f\para{x} = \frac{1}{2}\norm{x-y}{2}^2$, $g\equiv 0$, $h \equiv \iota_{\Ball_1^1}$ and $\C = \Ball_1^1$.

\begin{figure}[htbp]
\centering
\includegraphics[width=0.75\linewidth]{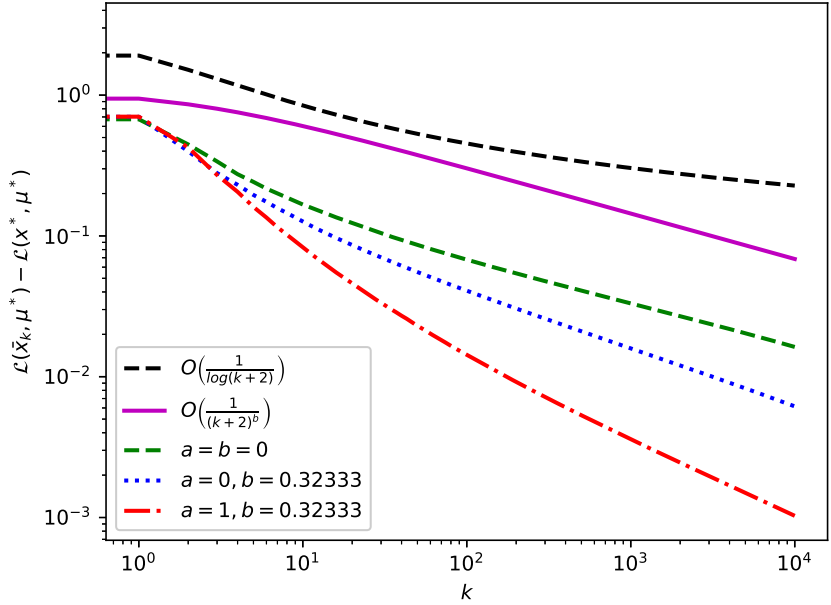}
\caption{Ergodic convergence profiles for \cgal applied to the simple projection problem.}
\label{fig:simple_convergence}
\end{figure}

The assumptions mentioned previously, i.e. \ref{ass:A1}-\ref{ass:coercive}, all hold in this finite-dimensional case as $f$, $g$, and $h$ are all in $\Gamma_0\para{\R^2}$, $f$ is Lipschitz-smooth, $h$ is the indicator function for a compact convex set, $g$ has full domain and $0 \in \ker(A) \cap \interop(\C)$. Regarding the parameters and the associated assumptions, we choose $\gamma_k$ according to Example~\ref{ex:seq} with $(a,b) \in \ens{(0,0),(0,1/3-0.01),(1,1/3-0.01)}$, $\theta_k = \gamma_k$, and $\rho = 2^{2-b}+1$. The ergodic convergence profiles of the Lagrangian are displayed in Figure~\ref{fig:simple_convergence} along with the theoretical rates (see \thmref{convergence} and Example~\ref{ex:seqconv}). The observed rates agree with the predicted ones of $O\para{\frac{1}{\log(k+2)}}$, $O\para{\frac{1}{(k+2)^{b}}}$ and $o\para{\frac{1}{(k+2)^{b}}}$ for the respective choices of $(a,b)$.

\subsection{Matrix completion problem} 
We also consider the following, more complicated matrix completion problem,
\nnewq{\label{eq:matcomp}
\min\limits_{X\in\R^{N \times N}}\enscond{\norm{\Omega X - y}{1}}{\norm{X}{*}\leq \delta_1, \norm{X}{1}\leq\delta_2},
}
where $\delta_1$ and $\delta_2$ are positive constants, $\Omega: \R^{N \times N}\to \R^p$ is a masking operator, $y\in\R^p$ is a vector of observations, and $\norm{\cdot}{*}$ and $\norm{\cdot}{1}$ are respectively the nuclear and $\ell^1$ norms. The mask operator $\Omega$ is generated randomly by specifying a sampling density, in our case $0.8$. We generate the vector $y$ randomly in the following way. We first generate a sparse vector $\tilde y \in \R^N$ with $N/5$ non-zero entries independently uniformly distributed in $[-1,1]$. We take the exterior product $\tilde y\tilde y^\top=X_0$ to get a rank-1 sparse matrix which we then mask to get $\Omega X_0$. The radii of the contraints in \eqref{eq:matcomp} are chosen according to the nuclear norm and $\ell^1$ norm of $X_0$, $\delta_1 = \frac{\norm{X_0}{*}}{2}$ and $\delta_2 = \frac{\norm{X_0}{1}}{2}$.

\subsubsection{\cgal}
Problem~\eqref{eq:matcomp} is a special instance of~\eqref{prod} with $n=2$, $f \equiv 0$, $g_i=\norm{\cdot-y}{1}/2$, $T_i=\Omega$, $h_1 = \iota_{\Ball_*^{\delta_1}}$, $h_2 = \iota_{\Ball_1^{\delta_2}}$, where $\Ball_*^{\delta_1}$ and $\Ball_1^{\delta_2}$ are the nuclear and $\ell^1$ balls of radii $\delta_1$ and $\delta_2$. We then follow the same steps as in \secref{subsec:nnonsmooth}. Let $\mc{H}_p = \R^{N \times N}$, $\pspace = \mc{H}_p^2$, $\bs{X} = \bp \comp{X}{1}\\\comp{X}{2}\ep\in\pspace$. We then have $G\para{\Omega\bs{X}} = \frac{1}{2}\para{\norm{\Omega\comp{X}{1}-y}{1}+\norm{\Omega\comp{X}{2}-y}{1}}$, and $H(\bs{X}) = \iota_{\Ball_*^{\delta_1}}(\comp{X}{1})+\iota_{\Ball_1^{\delta_2}}(\comp{X}{2})$. Then problem \eqref{eq:matcomp} is obviously equivalent to 
\nnewq{
\min\limits_{\bs{X}\in\pspace}\enscond{G\para{\Omega\bs{X}}+H(\bs{X})}{\PVo\bs{X}=0} ,
}
which is a special case of \eqref{prodequiv} with $F \equiv 0$. It is immediate to check that our assumptions \ref{ass:A1}-\ref{ass:coercive} hold. Indeed, all functions are in $\Gamma_0(\mc{H}_p)$ and $F \equiv 0$, and thus \ref{ass:A1} and \ref{ass:f} are verified. $\C = \Ball_*^{\delta_1} \times \Ball_1^{\delta_2}$ which is a non-mepty convex compact set. We also have $\Omega\C \subset \dom(\partial G)=\R^p \times \R^p$, and for any $\bs{z} \in \R^p \times \R^p$, $\partial G(\bs{z}) \subset \Ball_\infty^{1/2} \times \Ball_\infty^{1/2}$ and thus \ref{ass:interior} is verified. \ref{ass:lip} also holds with $L_h=0$. $\V$ is closed as we are in finite dimension, and thus \ref{ass:rangeA} is fulfilled. We also have, since $\dom(G\circ\Omega)=\pspace$,
\[
\bs{0} \in \V \cap \inte\para{\dom(G\circ\Omega)} \cap \inte\para{\C} = \V \cap \inte\pa{\Ball_*^{\delta_1}} \times \inte\pa{\Ball_1^{\delta_2}} ,
\]
which shows that \ref{ass:coercive} is verified. The latter is nothing but the condition in \cite[Fact~15.25(i)]{BAUSCHCOMB}. It then follows from the discussion in Remark~\ref{rem:assfun}\ref{rem:assfunproper} that \ref{ass:existence} holds true.

We use \algref{alg:CGAL} by choosing the sequence of parameters $\gamma_k = \frac{1}{k+1}$, $\beta_k = \frac{1}{\sqrt{k+1}}$, $\theta_k = \gamma_k$, and $\rho_k \equiv 15$, which verify all our assumptions \ref{ass:P1}-\ref{ass:cond} in view of Example~\ref{ex:seq}. Our choice of $\gamma_k$ is the most common in the literature, and it can be improved according to our discussion in the previous section. 

Finding the direction $\bs{S}_k$ by solving the linear minimization oracle is a separable problem, and thus each component is given by,
{\small
\nnewq{
\comp{S_k}{1} \in \Argmin\limits_{\comp{S}{1}\in \Ball_{\norm{\cdot}{*}}^{\delta_1}}
&\Bigg\langle\frac{\Omega^*\para{\Omega \comp{X_k}{1} - y - \prox_{\frac{\beta_k}{2}\norm{\cdot}{1}}\para{\Omega\comp{X_k}{1} - y}}}{\beta_k} \\
&\quad + \frac{1}{2}\para{\comp{\mu_k}{1}-\comp{\mu_k}{2} +\rho_k\para{\comp{X_k}{1}-\comp{X_k}{2}}}, \comp{S}{1}\Bigg\rangle,\\
\comp{S_k}{2} \in \Argmin\limits_{\comp{S}{2}\in \Ball_{\norm{\cdot}{1}}^{\delta_2}}
&\Bigg\langle\frac{\Omega^*\para{\Omega \comp{X_k}{2} - y - \prox_{\frac{\beta_k}{2}\norm{\cdot}{1}}\para{\Omega\comp{X_k}{2} - y}}}{\beta_k} \\
&\quad +\frac{1}{2}\para{\comp{\mu_k}{2}-\comp{\mu_k}{1} +\rho_k\para{\comp{X_k}{2}-\comp{X_k}{1}}}, \comp{S}{2}\Bigg\rangle .
}
}
Because of the structure of the sets $\Ball_{\norm{\cdot}{*}}^{\delta_1}$ and $\Ball_{\norm{\cdot}{1}}^{\delta_2}$, finding the first component of $\bs{S}_k$ reduces to computing the leading right and left singular vectors of 
\newq{
\frac{\Omega^*\para{\Omega \comp{X_k}{1} - y - \prox_{\frac{\beta_k}{2}\norm{\cdot}{1}}\para{\Omega\comp{X_k}{1} - y}}}{\beta_k} +\frac{1}{2}\para{\comp{\mu_k}{1}-\comp{\mu_k}{2} +\rho_k\para{\comp{X_k}{1}-\comp{X_k}{2}}}
}
while finding the second component reduces to computing the largest entry of 
{\small
\newq{
\left\lvert\para{\frac{\Omega^*\para{\Omega \comp{X_k}{2} - y - \prox_{\frac{\beta_k}{2}\norm{\cdot}{1}}\para{\Omega\comp{X_k}{2} - y}}}{\beta_k}+\frac{1}{2}\para{\comp{\mu_k}{2}-\comp{\mu_k}{1} +\rho_k\para{\comp{X_k}{2}-\comp{X_k}{1}}}}_{\para{i,j}}\right\rvert
}
}
over all the entries $\para{i,j}$. The dual variable update is given by,
\newq{
\bs{\mu}_{k+1} \eqdef \bp \comp{\mu_{k+1}}{1}\\ \comp{\mu_{k+1}}{2}\ep = \bp \comp{\mu_{k}}{1}\\ \comp{\mu_{k}}{2}\ep + \frac{\gamma_k}{2}\bp\comp{X_{k+1}}{1} - \comp{X_{k+1}}{2}\\ \comp{X_{k+1}}{2}-\comp{X_{k+1}}{1}\ep
}

\subsubsection{GFB}
Let $\mc{H}_p = \R^{N \times N}$, $\pspace = \mc{H}_p^3$, $\bs{W} = \bp \comp{W}{1}\\\comp{W}{2}\\\comp{W}{3} \ep \in \pspace$, $Q\para{\bs{W}} = \norm{\Omega\comp{W}{1}-y}{1} + \iota_{\Ball_{\norm{\cdot}{*}}^{\delta_1}}\para{\comp{W}{2}} + \iota_{\Ball_{\norm{\cdot}{1}}^{\delta_2}}\para{\comp{W}{3}}$. Then we reformulate problem \eqref{eq:matcomp} as
\nnewq{
\min\limits_{\bs{W}\in\pspace}\enscond{Q\para{\bs{W}}}{\bs{W}\in \V} ,
}
which fits the framework to apply the GFB algorithm proposed in~\cite{gfb2011} (in fact Douglas-Rachford since the smooth part vanishes).

The algorithm has three steps, each of which is separable in the components. We choose the step sizes $\lambda_k = \gamma = 1$ in the GFB to get,
{\small
\nnewq{
\begin{cases}
\bs{U}_{k+1}=\bp
2\comp{W_{k}}{1} - \comp{Z_{k}}{1} + 
\Omega^*\para{y-\Omega\para{2\comp{W_{k}}{1} - \comp{Z_{k}}{1}} +\prox_{\norm{\cdot}{1}}\para{\Omega\para{2\comp{W_{k}}{1} - \comp{Z_{k}}{1}}-y}}\\
\Pi_{\Ball_{\norm{\cdot}{*}}^{\delta_1}}\para{2\comp{W_{k}}{2} - \comp{Z_{k}}{2}}\\
\Pi_{\Ball_{\norm{\cdot}{1}}^{\delta_2}}\para{2\comp{W_{k}}{3} - \comp{Z_{k}}{3}}
\ep\\
\bs{Z}_{k+1}=\bs{Z}_{k} + \bs{U}_{k+1}-\bs{W}_{k}\\
\bs{W}_{k+1}=\bp
\sum_{i=1}^3\comp{Z_{k+1}}{i}/3\\
\sum_{i=1}^3\comp{Z_{k+1}}{i}/3\\
\sum_{i=1}^3\comp{Z_{k+1}}{i}/3
\ep
\end{cases}
}
}

We know from \cite{gfb2011} that $\bs{Z}_k$ converges to $\bs{Z}^\star$, and $\bs{W}_k$ and $\bs{U}_k$ both converge to $\bs{W}^\star=\Pi_{\V}(\bs{Z}^\star) = (X^\star,X^\star,X^\star)$, where $X^\star$ is a minimizer of \eqref{eq:matcomp}.

\subsubsection{Results}
We compare the performance of \cgal with GFB for varying dimension, $N$, using their respective ergodic convergence criteria. For \cgal this is the quantity $\mc{L}\para{\bar{\bs{X}}_k,\mu^*}-\mc{L}\para{\bs{X}^\star,\bs{\mu}^\star}$ where $\bar{\bs{X}}_k = \sum\limits_{i=0}^k \gamma_i \bs{X}_i/\Gamma_k$. Meanwhile, for GFB, we know from \cite{cesarefdr} that the Bregman divergence $D_Q^{\bs{v}^\star}\para{\bar{\bs{U}}_k}=Q(\bar{\bs{U}}_k)-Q(\bs{W}^\star)-\ip{\bs{v}^\star,\bar{\bs{U}}_k-\bs{W}^\star}{}$, with $\bar{\bs{U}}_k = \sum\limits_{i=0}^k\bs{U}_i/(k+1)$ and $\bs{v}^\star = (\bs{W}^\star - \bs{Z}^\star)/\gamma$, converges at the rate $O(1/(k+1))$. To compute the convergence criteria, we first run each algorithm for $10^5$ iterations to approximate the optimal variables ($\bs{X}^\star$ and $\bs{\mu}^\star$ for \cgal, and $\bs{Z}^\star$ and $\bs{W}^\star$ for GFB). Then, we run each algorithm again for $10^5$ iterations, this time recording the convergence criteria at each iteration. The results are displayed in Figure~\ref{fig:convergence}.

\begin{figure}[htbp]
\includegraphics[width=0.5\linewidth]{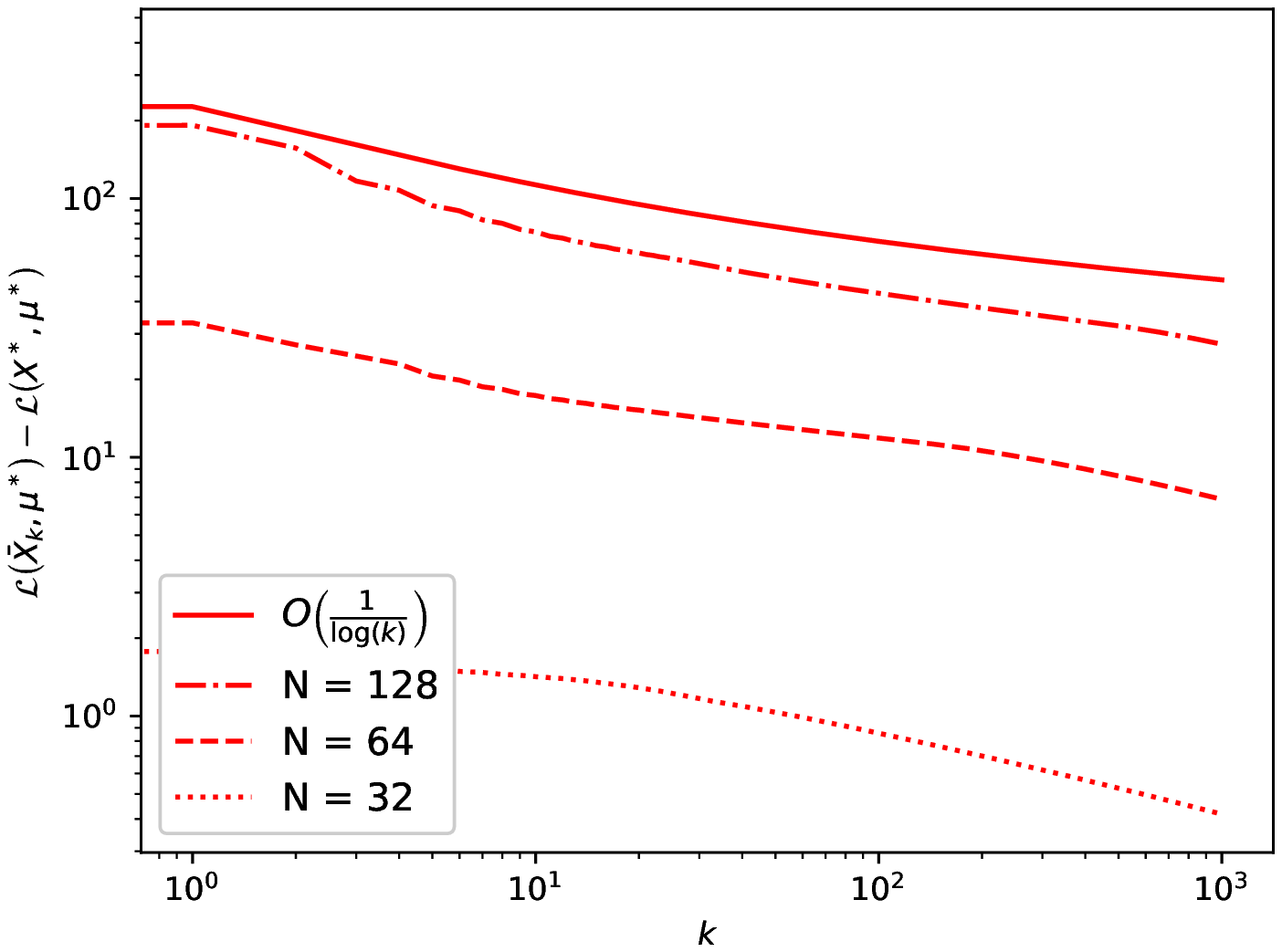}
\includegraphics[width=0.5\linewidth]{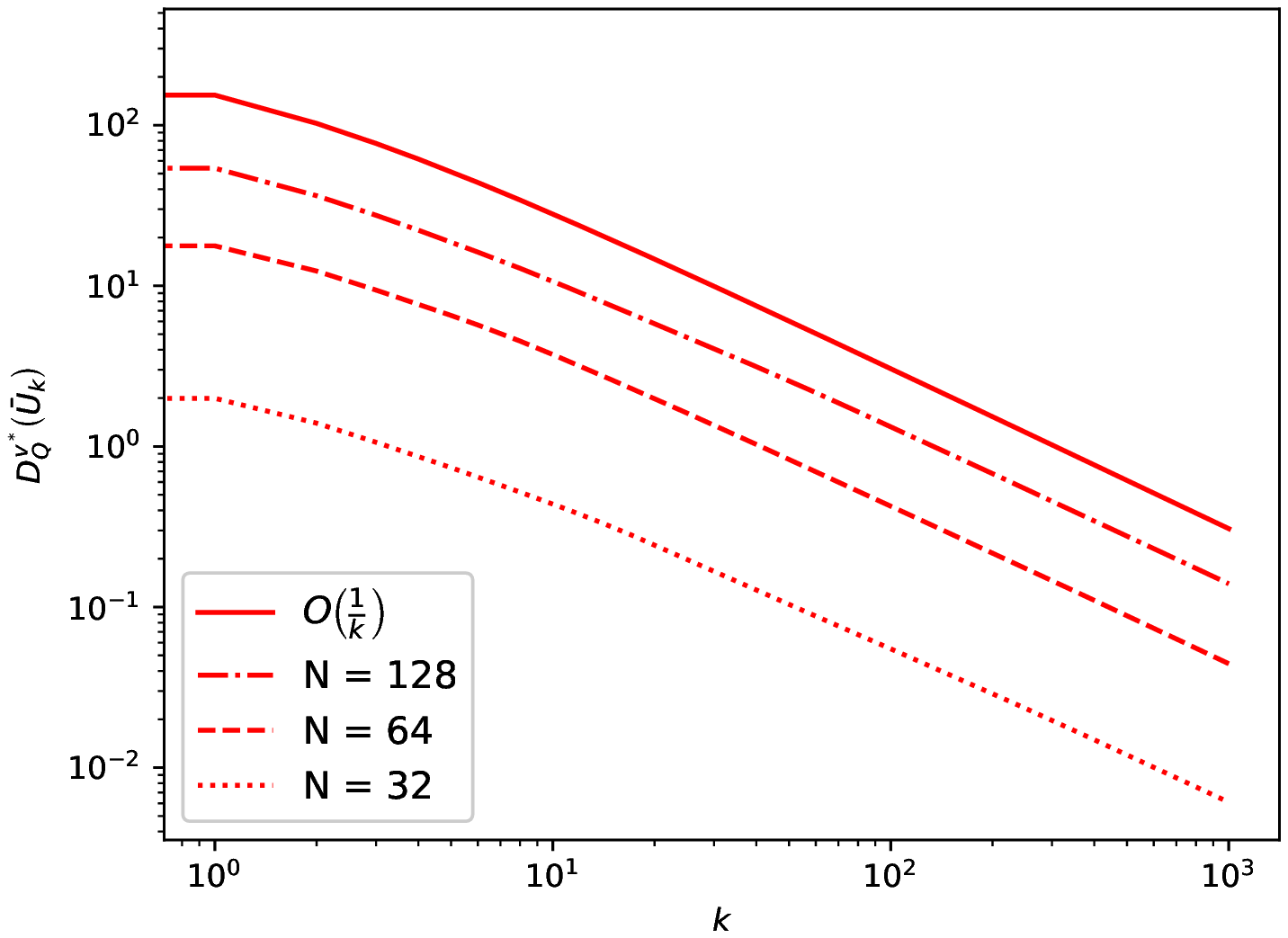}
\caption{Convergence profiles for \cgal (left) and GFB (right) for $N=32$, $N=64$, and $N=128$.}
\label{fig:convergence}
\end{figure}

It can be observed that our theoretically predicted rate (which is $O\para{1/\log(k+2)}$ for \cgal according to \thmref{convergence} and Example~\ref{ex:seqconv}) is in close agreement with the observed one. On the other hand, as is very well-known, employing a proximal step for the nuclear ball constraint will necessitate to compute an SVD which is much more time consuming than computing the linear minimization oracle for large $N$. For this reason, even though the rates of convergence guaranteed for \cgal are slower than for GFB, one can expect \cgal to be a more time computationally efficient algorithm for large $N$.


\section*{Acknowledgements}
ASF was supported by the ERC Consolidated grant NORIA. JF was partly supported by Institut Universitaire de France. CM was supported by Project MONOMADS funded by Conseil R\'egional de Normandie. We would like to warmly thank Gabriel Peyr{\'e} for his support and very fruitful and inspiring discussions.

\appendix
\small


\begin{small}
\bibliographystyle{plain}
\bibliography{ratesbib}
\end{small}

\end{document}